\documentclass[12pt]{amsart}
\usepackage{a4}
\usepackage[cp1250]{inputenc}
\usepackage{amssymb}
\usepackage[T1]{fontenc}
\usepackage{graphicx}

\usepackage{color}

\newcounter{liste}

\newenvironment{lister}{\begin{list}%
{{\rm(\alph{liste})\hfill}}{%
\usecounter{liste}%
\setlength{\labelsep}{0ex}%
\settowidth{\labelwidth}{(w)\ }%
\settowidth{\leftmargin}{(w)\ }%
\addtolength{\topsep}{1.4\topsep}%
}}{\end{list}}

\newenvironment{lister*}{\begin{list}%
{{\rm---\hfill}}{%
\setlength{\labelsep}{0ex}%
\settowidth{\labelwidth}{---}%
\setlength{\leftmargin}{\parindent}%
\addtolength{\topsep}{1.4\topsep}%
}}{\end{list}}

\newtheorem{statem}{}[section]

\newtheorem{rem}[statem]{Remark}

\newtheorem*{acknow}{Acknowledgement}

\newtheorem{prop}[statem]{Proposition}
\newtheorem{lemma}[statem]{Lemma}

\newtheorem{theo}[statem]{Theorem}
\newtheorem{defi}[statem]{Definition}
\newtheorem{cor}[statem]{Corollary}

\newtheorem*{prob}{Problem}
\newtheorem{quest}[statem]{Question}

\newcommand{\trind}{\mathop{\mathrm{trind}}\nolimits}
\newcommand{\trInd}{\mathop{\mathrm{trInd}}\nolimits}

\newcommand{\Ind}{\mathop{\mathrm{Ind}}\nolimits}
\newcommand{\ind}{\mathop{\mathrm{ind}}\nolimits}

\newcommand{\Indo}{\mathop{\mathrm{Ind}_0}\nolimits}
\newcommand{\cl}{\mathop{\mathrm{cl}}\nolimits}
\renewcommand{\int}{\mathop{\mathrm{int}}\nolimits}
\newcommand{\bd}{\mathop{\mathrm{bd}}\nolimits}

\newcommand{\card}{\mathop{\mathrm{card}}\nolimits}

\newcommand{\id}{\mathop{\mathrm{id}}\nolimits}

\newcommand{\dInd}[1]{\mathop{#1\mbox{\rm -}\mathrm{Ind}}\nolimits}
\newcommand{\ddim}[1]{\mathop{#1\mbox{\rm -}\mathrm{dim}}\nolimits}
\newcommand{\dIndo}[1]{\mathop{#1\mbox{\rm -}\mathrm{Ind}_0}\nolimits}
\newcommand{\dobs}[1]{\mathop{#1\mbox{\rm -}\mathrm{obs}}\nolimits}
\newcommand{\dstr}[1]{\mathop{#1\mbox{\rm -}\mathrm{str}}\nolimits}

\title[Dimensions modulo an ANR and modulo a simplicial complex]{On dimensions modulo a compact metric ANR and modulo a simplicial complex}

\author[J.\ Krzempek]{Jerzy Krzempek}

\address{Institute of Mathematics\\
Silesian University of Technology\\
Kaszubska 23\\
44-100 Gliwice\\
Poland}

\email{j.krzempek@polsl.pl}



\begin{document}

\maketitle

\begin{abstract}
V.\ V.\ Fedorchuk has recently introduced dimension functions $K$-dim~$\leq$ $K$-Ind and $L$-dim~$\leq$ $L$-Ind, where $K$~is a
simplicial complex and $L$~is a compact metric ANR. For each complex $K$ with a non-contractible join $|K|\ast |K|$ (we write $|K|$
for the geometric realisation of~$K$), he has constructed first countable, separable compact spaces with $K$-dim~$<$ $K$-Ind.

In a recent paper we have combined an old construction by P.\ Vop\v enka with a new construction by V.\ A.\ Chatyrko, and have
assigned a certain compact space $Z(X,Y)$ to any pair of non-empty compact spaces $X,Y$. In this paper we investigate the behaviour of
the four dimensions under the operation $Z(X,Y)$. This enables us to construct examples of compact Fr\'echet spaces which have
$K$-dim~$<$ $K$-Ind, $L$-dim~$<$ $L$-Ind, or $K$-Ind~$<$ $|K|$-Ind, and (connected) components of which are metrisable. In particular,
given a natural number $n\geq 1$, an ordinal $\alpha\geq n$, and any metric continuum $C$ with $\mathop{L\mbox{\rm -}\mathrm{dim}}
C=n$, we obtain
\begin{list}{}{\setlength{\labelsep}{0ex}\setlength{\labelwidth}{\parindent}\setlength{\leftmargin}{\parindent}\setlength{\topsep}{1ex}}
\item[$\bullet$\hfill] a compact Fr\'echet space $X_{C,\alpha}$ such that $L$-dim$\,X_{C,\alpha}=n$, $L$-Ind$\,X_{C,\alpha}=\alpha$,
and each component of $X_{C,\alpha}$ is homeomorphic to~$C$.
\end{list}
If $L\ast L$ is non-contractible, or $n=1$ and $L$ is non-contractible, then $C$ can be a cube $[0,1]^m$ for a certain natural number
$m=m(n,L)$.
\end{abstract}

\section*{Introduction}

{\em All considered topological spaces are\/ $T_1$ and completely regular.} Let $K$ be a~fixed (finite) simplicial complex, $|K|$ its
geometric realisation, and $L$ a (compact metric) ANR. {\em We assume that\/ $|K|$ and\/ $L$ are non-contractible}\/\footnote{If $|K|$
or $L$ were contractible, then the considered dimension functions would be trivial (they would assign zero to each non-empty normal
space).}.

V.\ V.\ Fedorchuk [\ref{fed09}--\ref{fed11}] has begun the investigation of dimensions\footnote{\label{foot}Note that Fedorchuk
\cite{fed-09,fed-10b} has defined $\ddim{\mathcal{K}}$, $\ddim{\mathcal{L}}$, $\dInd{\mathcal{K}}$, $\dInd{\mathcal{L}}$ for
collections $\mathcal{K}$ and~$\mathcal{L}$ which consist of simplicial complexes and ANR's, respectively. However, in the present
paper each of $\mathcal{K}$ and~$\mathcal{L}$ has exactly one element, $K$ or $L$, and we write $\ddim{K}$, $\ddim{L}$, $\dInd{K}$,
$\dInd{L}$.} $\ddim{K} X$, $\ddim{L} X$, $\dInd{K} X$, $\dInd{L} X$ of normal spaces~$X$. There is a far reaching analogy between the
theories of $\ddim{K}$/$\dInd{K}$, $\ddim{L}$/$\dInd{L}$, and the classical $\dim$/$\Ind$. In particular, $\ddim{K} X\leq\dInd{K} X$,
$\ddim{L}\leq\dInd{L} X$, $\ddim{K} X=\ddim{|K|} X$, and $\dInd{K} X\leq\dInd{|K|} X$ if $X$ is normal. Moreover $\dInd{K}
X=\dInd{|K|} X$ if $X$~is hereditarily normal, and all the four dimensions for $K$ and~$|K|$ coincide if $X$~is metrisable. In
\cite{fed-10b}, for each natural number $n\geq 2$ and each simplicial complex $K$ with a non-contractible join $|K|\ast|K|$, Fedorchuk
has constructed a first countable, separable compact space $X_n$ such that $\ddim{K} X_n=n<2n-1\leq\dInd{K} X_n\leq 2n$.

Henceforth, let $\dInd{K}$ and $\dInd{L}$ denote the transfinite extensions of Fe\-dor\-chuk's $\dInd{K}$ and $\dInd{L}$.

In the joint paper \cite{chk-12} with M.\ G.\ Charalambous we have constructed first countable and separable continua $S_{n,\alpha}$
such that $\ddim{K} S_{n,\alpha}=n$ and $\dInd{K} S_{n,\alpha}=\alpha$, where $n\geq 1$ is any natural number, $\alpha\geq n$ is any
ordinal of cardinality at most~$\mathfrak{c}$, and moreover $n=1$ or the join $|K|\ast |K|$ is non-contractible. This may be
considered as a partial solution to the following.

\begin{prob}
Let $n$ be a natural number, $\alpha$ an ordinal, and\/ $1\leq n\leq\alpha$.
\begin{lister}
\item Under what circumstances do there exist compact spaces with $\ddim{K}=n$ and $\dInd{K}=\alpha$?
\item Can all components of such a space be metrisable?
\item What about $\ddim{L}$ and $\dInd{L}$?
\end{lister}
\end{prob}

In \cite{krz-10a} we have combined constructions by P.\ Vop\v enka \cite{vop-58} and V.\ A.\ Chatyrko \cite{chat-08}, and have
assigned a compact space $Z(X,Y)$ to any pair of non-empty compact spaces $X,Y$. Each component of $Z(X,Y)$ is homeomorphic to a
component of $X$ or~$Y$. This has allowed us to construct compact Fr\'echet spaces $X_{C,\alpha}$ such that $\dim X_{C,\alpha}=n$,
$\trind X_{C,\alpha}=\trInd X_{C,\alpha}=\alpha$, and all components of $X_{C,\alpha}$ are homeo\-mor\-ph\-ic to~$C$, where $C$~is any
metric continuum with $\dim C=n<\infty$ and $\alpha\geq n$ is any ordinal (\cite[Theorem 5]{krz-10a}).

In the present paper we investigate the behaviour of Fedorchuk's dimensions under the operation $Z(X,Y)$. We prove that $\ddim{L}
Z(X,X)=\ddim{L} X$ (the same holds for $K$), and under mild assumptions on~$X$, $\dInd{L} Z(X,X)=\dInd{L} X+1$. We use transfinite
induction, and answer the questions (b, c) together by constructing examples of spaces satisfying $\ddim{L}=n$ and $\dInd{L}=\alpha$
in all cases without obvious obstructions (see the Abstract, Theorem \ref{ANR-main}, and Corollary \ref{ANR-cor}).

In the case of $\dInd{K}$ we encounter serious difficulties because it often happens that $\dInd{K} Z(X,X)=\dInd{K} X$. We distinguish
two sorts of spaces $X$ which satisfy the equality $\dInd{K} X=\alpha$ weakly or strongly, and we formalise this by defining the {\em
dimensional strength degree}\/ $\dstr{K} X\in\{0,1\}$. We confine ourselves to the case of $K=\partial\Delta^k$, the simplicial
complex that consists of the proper faces of a $k$-dimensional simplex~$\Delta^k$. We show that the spectrum of
$\dstr{\partial\Delta^k}$ on the class of compact metric spaces is $\{0,1\}$. We prove that if $\dstr{\partial\Delta^k} X=1$, then
$\dInd{\partial\Delta^k} Z(X,X)=\dInd{\partial\Delta^k} X +1$.

Our approach enables us to obtain the following examples. Let $C$ be a metric continuum and $n\geq 1$. Then there exists
\begin{list}{}{\setlength{\labelsep}{0ex}\setlength{\labelwidth}{\parindent}\setlength{\leftmargin}{\parindent}\addtolength{\topsep}{1.4\topsep}}
\item[$\bullet$\hfill] a compact Fr\'echet space $X_C$ with $\ddim{\partial\Delta^k} X_C=n$, $\dInd{\partial\Delta^k} X_C=n+1$,
and components homeomorphic to~$C$ {\em whenever $k\geq 1$ and\/ $\dim C=k(n+1)-1$} (in\linebreak this case $\ddim{\partial\Delta^k}
C=n$ and $\dstr{\partial\Delta^k} C=1$);
\item[$\bullet$\hfill] a compact Fr\'echet space $X_C$ such that
$\ddim{\partial\Delta^k} X_C=\dInd{\partial\Delta^k} X_C=n$ while $\dInd{|\partial\Delta^k|} X_C=n+1$, and each component of $X_C$ is
homeomorphic to~$C$---{\em this example needs the assumptions that $k\!\geq\! 2$ and\/ $\dim C\!=\!kn$} (then $\ddim{\partial\Delta^k}
C\!=\!n$ and $\dstr{\partial\Delta^k} C=0$).
\end{list}
Using the latter series of examples, we answer Fedorchuk's \cite[Question 3.1]{fed-10b} in the negative: the equality
$\dInd{K}=\dInd{|K|}$ is not true outside the class of hereditarily normal spaces.

\begin{acknow} \em
I sincerely thank Michael G.\ Charalambous for his contribution to joint work preceding this paper, for correspondence and apt
remarks, particularly an idea that made the original proof of Theorem \ref{zero-product} much simpler.
\end{acknow}

\section{Notation, basic definitions and facts}

In this paper {\em maps} and their {\em extensions} are meant to be continuous. A {\em continuum}\/ is a non-empty, connected compact
space. By $\mathbb{N}$ we denote the set of natural numbers, and $0\in\mathbb{N}$ is also the first ordinal. We write $A_\mathfrak{m}$
for the one-point compactification of the discrete space of cardinality $\mathfrak{m}$, and $\mu\in A_\mathfrak{m}$ is the unique
non-isolated point. In most cases we employ the terminology used in R.\ Engelking's\linebreak monographs \cite{eng-89,eng-95}.

We write $K$ for a (finite) simplicial complex with distinct vertices $e_0,\ldots,e_k$ in a Euclidean space, $|K|$ for the geometric
realisation of~$K$ (the underlying poly\-he\-dron), and $L$ for a (compact metric) ANR. {\em We assume that both\/ $|K|$ and\/ $L$ are
non-contractible.} $\Delta^k$ stands for the $k$-dimen\-sion\-al simplex with vertices $e_0,\ldots,e_k$, and $\partial\Delta^k$ for
the simplicial complex that consists of all at most $(k-1)$-dimen\-sion\-al faces of~$\Delta^k$. Of course, $|K|$ is always an~ANR,
and $|\partial\Delta^k|$ is homeomorphic to the sphere $S^{k-1}$. By $\{0,1\}$ we denote $\partial\Delta^1$, a simplicial complex that
has two vertices and no edge.

Let $X$ be a space, $A\subset X$, and $f\colon A\to L$ a map. An open set $U\subset X$ is called an {\em $L$-neighbourhood}\/ of~$f$
in~$X$ provided that $f$ has an extension from $U$ to~$L$. Then $P=X\setminus U$ is called an {\em $L$-partition}\/ in~$X$ for~$f$.
Since $L$~is an ANR, {\em every map $f\colon F\to L$ from a closed subset $F$ of a normal space $X$ has an $L$-neighbourhood and an
$L$-partition in~$X$.}

We adopt a convention, by which we use calligraphic letters $\mathcal{A}$, $\mathcal{B}$, etc.\ to denote {\em $(k+1)$-tuples}\/
$(A_0,\ldots,A_k)$, $(B_0,\ldots,B_k)$, etc.\ of subsets of any given space~$X$. A $(k+1)$-tuple $\mathcal{A}$ of $X$ is said to be
{\em open}\/ [respectively:\ {\em closed}\/] if $A_0,\ldots,A_k$ are open [respectively:\ closed] in~$X$. We write
$\cl\mathcal{A}=(\cl A_0,\ldots,\cl A_k)$, $\mathcal{A}|E=(A_0\cap E,\ldots,$ $A_k\cap E)$ for any subset $E$ of~$X$, $f(\mathcal{A})
=(f(A_0),\ldots,f(A_k))$ for a map $f$ defined on $X$, etc. A $(k+1)$-tuple $\mathcal{A}$ is called a {\em $K$-tuple}\/ provided that,
if $I\subset\{0,\ldots,k\}$ and $\bigcap_{i\in I} A_i\neq\emptyset$, then $\{e_i\colon i\in I\}$ is the vertex set of a certain
simplex in~$K$. We write $\bigcup\mathcal{A}=\bigcup_{i=0}^k A_i$. If $\mathcal{A}$~is an open $K$-tuple of $X$, we call
$P=X\setminus\bigcup\mathcal{A}$ the {\em $K$-partition corresponding to} $\mathcal{A}$; if moreover $B_i\subset A_i$ for
$i=0,\ldots,k$, we say that $\mathcal{A}$~is a {\em $K$-neigh\-bour\-hood of $\mathcal{B}$} and $P$~is a {\em $K$-partition
for~$\mathcal{B}$.}

We shall frequently use this simple corollary to \cite[Theorem 7.1.4]{eng-89}: {\em Every closed\/ $K$-tuple of a normal space has a\/ $K$-neighbourhood\/ $\mathcal{U}$ such that\/ $\cl\mathcal{U}$
is a $K$-tuple.}

\begin{defi}[Fedorchuk\footnote{See the remark in Footnote \ref{foot}.} {\cite[Definition 3.4]{fed-09}}] \label{def1}
\em Let $M$ be a simplicial complex. For normal spaces~$X$, the dimension $\ddim{M} X\!\in\!\mathbb{N}\cup\{-1,\infty\}$ is defined as
follows.
\begin{lister}
\item $\ddim{M} X=-1$ iff $X$ is empty.
\item When $n\in\mathbb{N}$, $\ddim{M} X\leq n$ iff for each sequence of closed $M$-tuples $\mathcal{F}^0,\ldots,\mathcal{F}^n$ of
$X$ there are $M$-partitions $P^i$ for $\mathcal{F}^i$, $i=0,\ldots,n$, so that $\bigcap_{i=0}^n P^i=\emptyset$.
\item $\ddim{M} X=\min\{n\in\mathbb{N}\colon\ddim{M} X\leq n\}$, where $X\neq\emptyset$ and $\min\emptyset=\infty$.
\end{lister}
\end{defi}

See Fe\-dor\-chuk \cite[Section 1]{fed-09} for information about the join $X\ast Y$ of compact spaces $X,Y$. At this place, let us
recall these two facts: $(X\ast Y)\ast Z$ is $X\ast (Y\ast Z)$ up to homeomorphism, and if $X,Y$ are ANR's, then so is $X\ast Y$. The
join $X\ast\ldots\ast X$ of $n$ copies of $X$ will be denoted by $X^{\ast n}$.

\begin{defi}[cf.\ Fedorchuk {\cite[Definition 3.9, Corollary 3.13]{fed-09}} and {\cite[Definition 1.14]{fed-10b}}]
\em Let $M$ be an ANR. Then $\ddim{M} X\in\mathbb{N}\cup\{-1,\infty\}$, where $X$ is any normal space, is defined so that it satisfies
the statements (a, c) of Definition \ref{def1} and the following statement (b') instead of~(b).
\begin{lister}
\item[{\rm(b')\hfill}] When $n\in\mathbb{N}$, $\ddim{M} X\leq n$ iff every map $f\colon F\to M^{\ast (n+1)}$ from a closed subspace $F$
of~$X$ has an extension from $X$ to $M^{\ast(n+1)}$.
\footnote{The author prefers the following equivalent statement.
\begin{list}{}{\setlength{\labelsep}{0ex}\settowidth{\labelwidth}{(b'')\ }\settowidth{\leftmargin}{(b'')\ }\setlength{\topsep}{1ex}}
\item[{\rm(b'')\hfill}]
When $n\in\mathbb{N}$, $\ddim{M} X\leq n$ iff for every $n+1$ maps $f_i\colon F_i\to M$, $i=0,\ldots,n$, where each $F_i\subset X$ is
closed, there is an $M$-partition $P^i$ in~$X$ for each $f_i$ such that $\bigcap_{i=0}^n P^i=\emptyset$ (cf.\ \cite[Theorem
1.31]{fed-10b} and Theorem \ref{gen-relations} herein).
\end{list}
However yet more, we prefer to avoid unnecessary analysis of the definition and extensions.}
\end{lister}
\end{defi}

\begin{defi}[cf.\ Fedorchuk {\cite[Definition 2.1]{fed-10b}} and {\cite[Definition 2.16]{fed-11}}] \label{def3}
\em Let $M$ be a simplicial complex. The inductive dimension\footnote{Fedorchuk's original $\dInd{K} X$ and $\dInd{L} X$ in
\cite{fed-10b} are natural numbers, $-1$, or $\infty$. Following \cite{chk-12}, we allow both $\dInd{K} X$ and $\dInd{L} X$ to be an
infinite ordinal.} $\dInd{M} X\!\in\! \mbox{\em Ordinals\/}\cup\{-1,\infty\}$ is defined for normal spaces $X$ as follows.
\begin{lister}
\item $\dInd{M} X=-1$ iff $X$ is empty.
\item When $\alpha$ is an ordinal, $\dInd{M} X\leq \alpha$ iff for every closed $M$-tuple $\mathcal{F}$ of $X$ there is an
$M$-partition~$P$ such that $\dInd{M} P<\alpha$.
\item $\dInd{M} X=\min\{\alpha\colon\dInd{M} X\leq \alpha\}$, where $X\neq\emptyset$ and $\min\emptyset=\infty$.
\end{lister}
\end{defi}

\begin{defi}[cf.\ Fedorchuk {\cite[Definition 2.3]{fed-10b}} and {\cite[Definition 2.14]{fed-11}}] \label{def4}
\em Let $M$ be an ANR. For normal spaces $X$, the dimension $\dInd{M} X\in \mbox{\em Ordinals\/}\cup\{-1,\infty\}$ is defined so that
it satisfies the statements (a, c) of Definition \ref{def3} and the following statement (b') instead of~(b).
\begin{lister}
\item[(b')\hfill] When $\alpha$ is an ordinal, $\dInd{M} X\leq \alpha$ iff for every map $f\colon F\to M$ from a closed subset $F$ of $X$
there is an $M$-partition~$P$ such that $\dInd{M} P<\alpha$.
\end{lister}
\end{defi}

It is evident that $\ddim{\{0,1\}} X=\dim X$ and $\dInd{\{0,1\}} X=\trInd X$ for normal spaces $X$, no matter whether we treat
$\{0,1\}$ as a simplicial complex or as an~ANR.

Let us recall the following well-known facts on homotopy equivalence.

\begin{theo}[J.\ E.\ West \cite{wes-77}] \label{west}
Every compact metric ANR is homotopy equivalent to a compact polyhedron.\qed
\end{theo}

\begin{theo}[Fedorchuk {\cite[Proposition 4.5]{fed-09}} and {\cite[Theorem 3.3]{fed-11}}] \label{ANR-eqiv-eq}
If ANR's $L_1$ and $L_2$ are homotopy equivalent, then
$$\ddim{L_1} X=\ddim{L_2} X \mbox{ \ \ \ and \ \ \ } \dInd{L_1} X=\dInd{L_2} X$$
for every normal space $X$.\qed
\end{theo}

It follows from the foregoing two theorems that when we investigate relations between the four dimensions $\ddim{K}$, $\ddim{L}$,
$\dInd{K}$, and $\dInd{L}$, it is sufficient to consider only simplicial complexes $K$ and their geometric realisations $L=|K|$.

\begin{theo}[Fedorchuk \cite{fed-09,fed-10b,fed-11}] \label{gen-relations}
Suppose that $X$ is a normal space. Then
$$\ddim{K} X=\ddim{|K|} X \mbox{ \ \ \ and \ \ \ }\dInd{K} X\leq\dInd{|K|} X.$$

If $X$ is hereditarily normal, then
$$\dInd{K} X=\dInd{|K|} X.$$

If either $\ddim{K} X$ or $\dInd{K} X$ is finite, then $$\ddim{K} X\leq\dInd{K} X.$$

If $X$ is metrisable and $\ddim{K} X$ is finite, then all four of the dimensions of $X$ coincide.
\end{theo}

\begin{proof}[Remarks on proofs.]
All of the inequalities and equalities have been shown by Fedorchuk. The equality $\ddim{K} X=\ddim{|K|} X$ is \cite[Theorem
4.8]{fed-09}.

The inequality $\dInd{K} X\leq\dInd{|K|} X$ is \cite[(3.1)]{fed-10b} and \cite[Theorem 2.22]{fed-11}. One can also prove it easily
using \cite[Lemma 6]{chk-12} and induction on $\alpha=\dInd{|K|}X$.

The inequality $\dInd{|K|} X\leq\dInd{K} X$ for hereditarily normal spaces $X$ is in \cite[Theorem 2.22]{fed-11} (cf.\ also
\cite[Theorem 2.4]{fed-10b}). To prove it, one can also apply \cite[Lemma 7]{chk-12} (Lemma \ref{extend} herein) and induction on
$\alpha=\dInd{K} X$.

The other assertions are \cite[Theorems 3.18 and 3.23]{fed-10b}.
\end{proof}

The topic of dimension-lowering maps for $\ddim{K}$ and $\ddim{L}$ is more complex than in the case of $\dim$ (see \cite[Section
7]{fed-09}). However, there is

\begin{theo}[cf.\ Fedorchuk {\cite[Theorem~3.24]{fed-10b}}] \label{compon1}
If $X$ is a compact space, then
$$\ddim{L} X=\sup\{\ddim{L} P\colon P\mbox{ is a component of } X\}.$$
\end{theo}

\begin{proof}
By Theorems \ref{west} and \ref{ANR-eqiv-eq}, it is sufficient to consider $L=|K|$. Theorem~\ref{gen-relations} yields the equalities
$\ddim{|K|} X=\ddim{K} X$ and $\ddim{|K|} P=\ddim{K} P$ for each component $P$ of $X$. Consider the decomposition $\mathcal{D}$ of $X$
into the components of~$X$ and the quotient map $q\colon X\to X/\mathcal{D}$. The quotient space $X/\mathcal{D}$ is compact and $\dim
X/\mathcal{D}=0$ unless $X$ is empty. The requested equality results from Fedorchuk's \cite[Theorem 3.24]{fed-10b} applied to~$q$.
\end{proof}

\begin{theo}\label{equiv}
Suppose that\/ $k,n\geq 1$ are natural numbers, and\/ $X$ is a metric space. Then
$$\dInd{\partial\Delta^{k}} X<n \mbox{ \ \ \ iff \ \ \ }\ddim{|\partial\Delta^{k}|} X<n \mbox{ \ \ \ iff \ \ \ }\dim X<kn.$$
\end{theo}

\begin{proof}
The former equivalence results from Theorem \ref{gen-relations}. The latter for $n=1$ is the well-known theorem on extending maps to
spheres (see \cite[Theorem 3.2.10]{eng-95}).

We shall apply this theorem by Fedorchuk \cite[Theorem 5.7 and Corollary 5.16]{fed-09}: {\em A~metric space $X$ has $\ddim{L} X\leq
n\in\mathbb{N}$ iff there are subspaces $X_0,\ldots,X_n$ of~$X$ such that $X=X_0\cup\ldots\cup X_n$ and $\ddim{L} X_i\leq 0$ for
$i=0,\ldots,n$.}

Let $n>1$ and $L=|\partial\Delta^k|$. Then $\ddim{|\partial\Delta^{k}|} X<n$ iff $X=X_0\cup\ldots\cup X_{n-1}$ and
$\ddim{|\partial\Delta^{k}|} X_i\leq 0$ for $i=0,\ldots,n-1$. These last inequalities are equivalent to $\dim X_i<k$, and in turn, to
the statement that $X_i=X_i^0\cup\ldots\cup X_i^{k-1}$ and $\dim X_i^j\leq 0$ for $j=0,\ldots,k-1$ (by \cite[Theorem 4.1.17]{eng-95}).
Thus, $\ddim{|\partial\Delta^{k}|} X<n$ iff $X$ is the union of at most $kn$ subspaces $X_i^j$ with $\dim X_i^j=0$, i.e.\ iff $\dim
X<kn$ (again by \cite[Theorem 4.1.17]{eng-95}).
\end{proof}

Suppose that $\mathcal{U}$ is an open $K$-tuple of a space $X$. We say that an element $x\in X$ is a {\em $K$-obstruction point for}
$\mathcal{U}$ provided that $\mathcal{U}$ has no $K$-neighbourhood~$\mathcal{V}$ with $x\in\bigcup\mathcal{V}$. We write
$\dobs{K}\mathcal{U}$ for the set of $K$-obstruction points for~$\mathcal{U}$. Clearly, $\dobs{K}\mathcal{U}$ does not intersect
$\bigcup\mathcal{U}$.

Let us note the following simple observation.

\begin{lemma}\label{obs}
Consider $K=\partial\Delta^k$. Then
$$\dobs{\partial\Delta^k}\mathcal{U}={\bigcap}_{\,0\leq i\leq k}\cl\left({\bigcap}_{\,0\leq j\leq k,\,j\neq i} U_j\right)$$
for every open $\partial\Delta^k$-tuple $\mathcal{U}=(U_0,\ldots,U_k)$.
\end{lemma}

\begin{proof}
Assume that $x\not\in\dobs{\partial\Delta^k}\mathcal{U}$, i.e.\ $x\in\bigcup\mathcal{V}$ for a certain
$\partial\Delta^k$-neighbourhood $\mathcal{V}=(V_0,\ldots,V_k)$ of~$\mathcal{U}$. If $x\in V_i$, then $x\not\in\cl({\bigcap}_{\,0\leq
j\leq k,\,j\neq i} U_j)\subset\cl({\bigcap}_{\,0\leq j\leq k,\,j\neq i} V_j)$ as $\mathcal{V}$~is a $\partial\Delta^k$-tuple. Thus,
$x$ does not belong to the intersection of closures.

Assume there is an $i$ such that $x\not\in\cl({\bigcap}_{\,0\leq j\leq k,\,j\neq i} U_j)$. Then there is a neighbourhood $W\ni x$
disjoint from ${\bigcap}_{\,0\leq j\leq k,\,j\neq i} U_j$. The union $V_i=U_i\cup W$ and the sets $V_j=U_j$, $j\neq i$, form a
$\partial\Delta^k$-neighbourhood $\mathcal{V}$ of~$\mathcal{U}$, and hence, $x\not\in\dobs{\partial\Delta^k}\mathcal{U}$.
\end{proof}

Considering the dimension $\dInd{K}$, we distinguish two ways, in which a space~$X$ may be $\alpha$-di\-men\-sion\-al. We define the
{\em dimensional strength degree} $\dstr{K} X\!\in\!\{0,1\}$ as follows. Let $0<\alpha=\dInd{K} X<\infty$. We put $\dstr{K} X=0$
($X$~is {\em weakly}\/ $\alpha$-dimensional) when every closed $K$-tuple of~$X$ has a $K$-neigh\-bour\-hood $\mathcal{U}$ with
$\dobs{K}\mathcal{U}=\emptyset$ and $\dInd{K}(X\setminus\bigcup\mathcal{U})<\alpha$. Otherwise, we put $\dstr{K} X=1$ (i.e.\ $X$~is
{\em strongly}\/ $\alpha$-dimensional when $0<\alpha=\dInd{K} X<\infty$ and there is a closed $K$-tuple whose every
$K$-neigh\-bour\-hood~$\mathcal{U}$ with $\dInd{K}(X\setminus\bigcup\mathcal{U})<\alpha$ has $\dobs{K}\mathcal{U}\neq\emptyset$). By
abuse of notation, we write $\dstr{K}X=0$ when $\alpha$ is $-1$, $0$, or $\infty$.

In the next section we prove that the above distinction is material at least for some $K$'s: if $1\leq n\in\mathbb{N}$ and $2\leq
k\in\mathbb{N}$, then---for instance---the following cubes have $\dInd{\partial\Delta^k} [0,1]^{kn}=\dInd{\partial\Delta^k}
[0,1]^{k(n+1)-1}=n$, $\dstr{\partial\Delta^k} [0,1]^{kn}=0$, and $\dstr{\partial\Delta^k}[0,1]^{k(n+1)-1}=1$ (cf.\ Theorem \ref{equiv}
and Propositions \ref{big-space1}--\ref{big-space2}). On the other hand, in the case when $k=1$ and $K=\{0,1\}$, {\em every normal
space $X$ with\/ $\trInd X$ being a successor ordinal has\/ $\dstr{\{0,1\}} X=1$.} Indeed, let $\alpha=\trInd X>0$, and suppose on the
contrary that $\dstr{\{0,1\}} X=0$. Take arbitrary disjoint closed sets $F_0,F_1\subset X$. Then, the $\{0,1\}$-tuple $(F_0,F_1)$ has
a $\{0,1\}$-neighbourhood $(U_0,U_1)$ with $P=X\setminus (U_0\cup U_1)$, $\trInd P\leq\alpha-1$, and $\dobs{\{0,1\}}(U_0,U_1)=\cl
U_0\cap\cl U_1=\emptyset$. Hence, there exists a partition~$Q$ between $\cl U_0$ and $\cl U_1$ with $\trInd Q<\alpha-1$, and we have
shown that $\trInd X\leq\alpha-1$. A contradiction. Therefore $\dstr{\{0,1\}} X=1$. Finally, the Smirnov compactum $S_{\omega_0}$
(i.e.\ the one-point compactification of the topological sum $\bigoplus_{i=1}^\infty [0,1]^i$) has $\trInd S_{\omega_0}=\omega_0$ and
$\dstr{\{0,1\}} S_{\omega_0}=0$.

Using the definition of $\dstr{K}$, one easily proves the following.

\begin{prop}\label{subspace}
Suppose that $A$ is a closed subspace of a normal space~$X$. If $\dInd{K}A=\dInd{K} X$ and $\dstr{K} X=0$, then $\dstr{K} A=0$.\qed
\end{prop}

\section{General lemmas}\label{general}

In this section we collect miscellaneous properties of Fedorchuk's dimensions, prove a combinatorial analogue (Corollary
\ref{dim-partition}) of Yu.\ T.\ Lisitsa's theorem \cite{lis-79} on partial extensions of maps into spheres (Theorem \ref{lis}
herein), investigate the $\dstr{\partial\Delta^k}$ of metric spaces, and prove the theorem on the dimensions of a product with a
compact discontinuum.

\begin{prop}[cf.\ Fedorchuk {\cite[Theorem 2.5]{fed-10a}}, Charalambous and Krzempek {\cite[Corollary 2]{chk-12}}] \label{cube}
Let $n\geq 1$ be a natural number. If $n=1$ or the join $L\ast L$ is non-contractible, then there is a natural number $m$ such that
$\ddim{L}[0,1]^m=n$.
\end{prop}

\begin{proof}
If $L\ast L$ is non-contractible, we are done by Fedorchuk's \cite[Theorem 2.5]{fed-10a}. Assume that $n=1$ and $L\ast L$ is
contractible. Then every normal space $X$ has $\ddim{L} X\leq 1$ by Fedorchuk's \cite[Proposition 2.3]{fed-10a}. In view of Theorems
\ref{west} and \ref{ANR-eqiv-eq}, it is sufficient to consider $L=|K|\subset [0,1]^{m-1}$. As $|K|$ is non-contractible, a certain map
from $|K|\times\{0,1\}$ to $|K|$ does not have an extension from $|K|\times [0,1]$ to $|K|$. Therefore, $0<\ddim{|K|} (|K|\times
[0,1])\leq\ddim{|K|}[0,1]^m\leq 1=n$.
\end{proof}

The following lemma is an $\dInd{L}$ analogue of \cite[Proposition 1]{chk-12}.

\begin{lemma}\label{zero-Dowker}
Let $X$ be a normal space, and $F\subset X$ be closed. If $\dInd{L} F=0$ and $\dInd{L} E\leq\alpha$ for each closed subset $E\subset
X$ disjoint from~$F$, then $\dInd{L} X\leq\alpha$.
\end{lemma}

\begin{proof}
Take any map $g\colon G\to L$, where $G\subset X$ is closed. Since $\dInd{L} F=0$, we infer that $g$~has an extension from $G\cup F$
to $L$. As $L$ is an ANR, we now obtain a neighbourhood $U$ of $G\cup F$ with an extension $g'\colon U\to L$ of~$g$. Let $V\subset X$
be an open set with $G\cup F\subset V\subset \cl V\subset U$. Since $\dInd{L} (X\setminus V)\leq\alpha$, there is an $L$-partition $P$
in $X\setminus V$ for the restriction $g'|\bd V$, where $\dInd{L} P<\alpha$. This means that $P\subset X\setminus\cl V$, and $g'|\bd
V$ has an extension $g''\colon X\setminus (V\cup P)\to L$. Finally, $(g'|\cl V)\cup g''\colon X\setminus P\to L$ extends $g$, and we
have shown that $\dInd{L} X\leq\alpha$.
\end{proof}

Recall that any $x\in |K|$ can be uniquely written in the form $x=\sum_{i=0}^k x_i e_i$, where the {\em barycentric coordinates\/
$x_0,\ldots,x_k$} are non-negative real numbers with $\sum_{i=0}^k x_i =1$. Put $K_i=\{x\in |K|\colon x_i\geq\frac{1}{k+1}\}$, and
note that $\mathcal{K}=(K_0,\ldots,K_k)$ is a closed $K$-cover of $|K|$.

\begin{lemma}[{\cite[Lemma 7]{chk-12}}]\label{extend}
Suppose that $f\colon  F \to |K|$ is a map from a closed subset $F$ of a normal space~$X$. If the $K$-tuple $f^{-1}(\mathcal{K})$ has
a $K$-neigh\-bour\-hood that covers $X$, then $f$ has an extension from $X$ to~$|K|$.\qed
\end{lemma}

\begin{theo}[Lisitsa \cite{lis-79}; see also {\cite[Problem 1.9.D]{eng-95}}]\label{lis}
Let $k\geq 1$, $m\geq -1$ be integers, and $X$ a normal space. If each map $f\colon F\to S^{k-1}$ from any closed subset $F$ of~$X$
has an extension from $X\setminus P$ to $S^{k-1}$, where $P\subset X$ is closed, does not meet $F$, and\/ $\dim P\leq m$, then\/ $\dim
X\leq k+m$.\qed
\end{theo}

\begin{cor}\label{dim-partition}
Let $k\geq 1$, $m\geq -1$ be integers, and $X$ a normal space. If~every closed $\partial\Delta^{k}$-tuple of~$X$ has a
$\partial\Delta^{k}$-partition $P$ such that\/ $\dim P\leq m$ and the complement $X\setminus P$ is a normal space, then\/ $\dim X\leq
k+m$.
\end{cor}

\begin{proof}
In order to use Lisitsa's theorem, take a map $f\colon F\to |\partial\Delta^k|$, where $F\subset X$ is closed. Consider the closed
$\partial\Delta^k$-cover $\mathcal{F}=f^{-1}(\mathcal{K})$ of~$F$. If $\mathcal{F}$~has a $\partial\Delta^{k}$-neigh\-bour\-hood
$\mathcal{U}$ such that the corresponding $\partial\Delta^{k}$-partition $P=X\setminus\bigcup\mathcal{U}$ satisfies the inequality
$\dim P\leq m$ and $U=\bigcup\mathcal{U}$ is normal, then $f$~extends to a map $f'\colon U\to |\partial\Delta^{k}|$ by Lemma
\ref{extend}. Therefore, $\dim X\leq k+m$ by Lisitsa's theorem.
\end{proof}

It is clear why the extension Lemma \ref{extend} and the upper bound of the covering dimension in Theorem \ref{lis} need a normality
assumption. The natural range of applications of Corollary \ref{dim-partition} is the class of hereditarily normal spaces. In view of
\cite[Lemma 6]{chk-12}, the corollary implies Lisitsa's theorem for any hereditarily normal space~$X$. They both should be compared
with \cite[Problems 2.2.B]{eng-95}---it is easily checked that they all three together imply Theorem \ref{equiv}. We do not know {\em
if either the hereditary normality or the normality of the complement in the corollary is a necessary assumption.}

\begin{lemma}\label{smaller-partition}
Suppose that $X$ is a metric space, $\mathcal{U}$ is an open $K$-tuple of $X$, and $P=X\setminus\bigcup\mathcal{U}$ is the
corresponding partition. If\/ $\Ind\dobs{K}\mathcal{U}<\Ind P\in\mathbb{N}$, then $\mathcal{U}$ has a $K$-neighbourhood whose
corresponding partition $Q$ has $\Ind Q<\Ind P$.
\end{lemma}

\begin{proof}
Write $m=\Ind P$. At first, we shall prove the lemma under the assumption that $\dobs{K}\mathcal{U}=\emptyset$. Then, let
$$W_i=\bigcup\{V_i\colon \mathcal{V}=(V_0,\ldots,V_k)\mbox{ is a $K$-neighbourhood of $\mathcal{U}$.}\}$$ for $i=0,\ldots,k$. Since
$\dobs{K}\mathcal{U}=\emptyset$, the sets $W_i$ form an open cover of~$X$, and the cover has a closed shrinking that consists of sets
$F_i\subset W_i$. For each $i$, there exists an open set $W_i'$ such that $F_i\subset W_i'\subset\cl W_i'\subset W_i$ and $\Ind
(P\cap\bd W_i')<m$ (\cite[Theorem 4.1.13]{eng-95}). Let $W_0''=W_0'$ and $W_i''=W_i'\setminus \cl(W_0'\cup\ldots\cup W_{i-1}')$ for
$0<i\leq k$. From the two facts that the sets $W_i'$ cover $X$ and $\bd W_i''\subset\bd W_0'\cup\ldots\cup\bd W_i'$, we infer that
$Q=P\setminus (W_0''\cup\ldots\cup W_k'')\subset P\cap(\bd W_0'\cup\ldots\cup \bd W_k')$. We obtain $\Ind Q<m$ by the countable sum
theorem (\cite[Theorem 4.1.9]{eng-95}). As easily checked, the unions $V_i=U_i\cup W_i''$ form a $K$-neigh\-bour\-hood $\mathcal{V}$
of~$\mathcal{U}$, and $Q=X\setminus\bigcup\mathcal{V}$.

Assume that $\Ind\dobs{K}\mathcal{U}<m$. Let $X_0=X\setminus\dobs{K}\mathcal{U}$\/ and $P_0=P\setminus\dobs{K}\mathcal{U}$. Then,
$\mathcal{U}$ has no $K$-obstruction points in $X_0$, and by the first part of the proof, there exists a $K$-neighbourhood
$\mathcal{V}$ of $\mathcal{U}$ in $X_0$ with the corresponding $K$-partition $Q_0=X_0\setminus\bigcup\mathcal{V}$ and $\Ind Q_0<m$.
Now, $Q=Q_0\cup\dobs{K}\mathcal{U}$ corresponds to $\mathcal{V}$ in~$X$, and $\Ind Q<m$ by the countable sum theorem.
\end{proof}

The foregoing lemma is also true when $X$ is a {\em strongly hereditarily normal space}\/ (see \cite[Definition 2.1]{eng-95}). To
prove this, one applies \cite[the statements 2.2.4, 2.3.6 and 2.3.7]{eng-95} instead of theorems on dimension in the class of metric
spaces.

\begin{prop}\label{big-space1}
Let $k\!\geq\! 1$ and $m\!\geq\! 0$. If $X$ is a metric space with\/ $\dim X\!\geq\! k+m$, then there exists a closed
$\partial\Delta^k$-tuple $\mathcal{F}$ of $X$ such that every $\partial\Delta^k$-neigh\-bour\-hood $\mathcal{U}$ of $\mathcal{F}$
satisfies the following alternative: $\dim\dobs{\partial\Delta^k}\mathcal{U}=m$ or the corresponding\linebreak
$\partial\Delta^k$-partition $P=X\setminus\bigcup\mathcal{U}$ has $\dim P>m$.

In particular, if $n\geq 1$ and\/ $\dim X=k(n+1)-1$, then $\dInd{\partial\Delta^k} X= n$ and\/ $\dstr{\partial\Delta^k} X=1$.
\end{prop}

\begin{proof}
Let $X$ be metric, and $\dim X\geq k+m$. By Corollary \ref{dim-partition}, there is a closed $\partial\Delta^k$-tuple $\mathcal{F}$
whose every $\partial\Delta^k$-neigh\-bour\-hood $\mathcal{U}$ has $\dim (X\setminus\bigcup\mathcal{U})\geq m$. Thus, if
$P=X\setminus\bigcup\mathcal{U}$ and $\dim P=m$, then $\dim\dobs{\partial\Delta^k}\mathcal{U}=m$ by Lemma \ref{smaller-partition}.

If $\dim X=k(n+1)-1$, then Theorem \ref{equiv} implies that $\dInd{\partial\Delta^k} X= n$. For $m=kn-1\geq 0$, let $\mathcal{F}$ be a
closed $\partial\Delta^k$-tuple whose every $\partial\Delta^k$-neighbourhood $\mathcal{U}$ satisfies the stated alternative. If
$P=X\setminus\bigcup\mathcal{U}$ has $\dInd{\partial\Delta^k} P<n$, then $\dim P\leq kn-1$ by Theorem \ref{equiv}, and
$\dim\dobs{\partial\Delta^k}\mathcal{U}=kn-1$. This means that $\dstr{\partial\Delta^k} X=1$.
\end{proof}

\begin{prop}\label{big-space2}
Let $k\geq 2$ and $n\geq 1$. If\/ $X$ is a metric space with\/ $\dim X=kn$, then $\dInd{\partial\Delta^k} X= n$ and\/
$\dstr{\partial\Delta^k} X=0$.
\end{prop}

\begin{proof}
If $X$ is metric and $\dim X=kn$, then $\dInd{\partial\Delta^k} X= n$ by Theorem \ref{equiv}.

Take a closed $\partial\Delta^{k}$-tuple $\mathcal{F}$ of~$X$, and find an open $\partial\Delta^k$-neigh\-bour\-hood $\mathcal{V}$
of~$\mathcal{F}$. There is an open set $W$ with $\bigcup\mathcal{F}\subset W\subset\cl W\subset \bigcup\mathcal{V}$ and $\dim\bd
W<kn$. Put $U_i=V_i\cap W$ for $i=0,\ldots,k-1$, $U_k=(V_k\cap W)\cup (X\setminus\cl W)$, and $P=\bd W$. Using Lemma \ref{obs} and the
inequality $k\geq 2$, one easily checks that $\dobs{\partial\Delta^k}\mathcal{U}=\emptyset$ for $\mathcal{U}=(U_0,\ldots,U_k)$.
Finally, we obtain $\dInd{\partial\Delta^k}P<n$ by Theorem \ref{equiv}. Therefore, $\dstr{\partial\Delta^k} X=0$.
\end{proof}

In Propositions \ref{big-space1}--\ref{big-space2} we have shown that if $k\geq 2$ and $1\leq n\in\mathbb{N}$, then there are two
degrees to which a compact metric space $X$ may have $\dInd{\partial\Delta^k} X=n$. Maybe there are more such (similar) degrees, but
at this moment we have neither good motivation nor good examples, which could help us to identify and point out appropriate
combinatorial properties of spaces in terms of $K$-neighbourhoods and $K$-partitions.

The formulas (a) and (c) in the next statement are generalisations of P.~Vop\v en\-ka's theorem \cite[p.\ 320]{vop-58} on the
classical~$\Ind$.

\begin{theo}\label{zero-product}
If $X$ and $Y$ are compact spaces and\/ $\dim X=0$, then
\begin{lister}
\item $\dInd{K} (X\times Y)=\dInd{K} Y$,
\item $\dstr{K} (X\times Y)=\dstr{K} Y$, and
\item $\dInd{L} (X\times Y)=\dInd{L} Y$.
\end{lister}
\end{theo}

\begin{proof}
(a) Evidently $\dInd{K} (X\times Y)\geq\dInd{K} Y$. We prove ``$\leq$'' by induction on $\alpha=\dInd{K} Y$. If $\alpha=-1$, we are
done. Assume that $\alpha\geq 0$. Write $\pi_X\colon X\times Y\to X$ and $\pi_Y\colon X\times Y\to Y$ for the projections. Take a
closed $K$-tuple $\mathcal{F}$ of $X\times Y$. For any point $x\in X$ consider the sets $\pi_Y(F_i\cap\pi_X^{-1}(x))\subset Y$,
$i=0,\ldots,k$. For this closed $K$-tuple of $Y$, take a $K$-neigh\-bour\-hood $\mathcal{U}^{\,x}$ whose corresponding
$K$-partition~$P^{\,x}$ has $\dInd{K} P^{\,x}< \alpha$. Since $Y$~is compact, $\pi_X$ is a closed map and the image $\pi_X
(F_i\setminus (X\times U_i^x))\not\ni x$ is a closed subset of $X$ for each~$i$. Hence, there is a neighbourhood $V^x\ni x$ such that
$F_i\cap\pi_X^{-1}(V^x)\subset X \times U^x_i$ for each $i$. Take a finite clopen refinement $\{W^s\colon s\in S\}$ of $\{V^x\colon
x\in X\}$ consisting of disjoint sets. For each~$s$ fix a point~$x_s$ with $W^s \subset V^{x_s}$. We have
$F_i\cap\pi_X^{-1}(W^s)\subset W^s \times U^{x_s}_i$ for each $i$ and $s$. The sets $U_i=\bigcup_{s\in S} W^s\times U_i^{x_s}$,
$i=0,\ldots,k$, form a $K$-neighbourhood $\mathcal{U}$ of~$\mathcal{F}$. Note the fact, which will be needed in a while, that
\begin{lister}
\item[$(\ast)$\hfill] \em if $\dobs{K}\mathcal{U}^{\,x_s}=\emptyset$ for each $s\in S$, then $\dobs{K}\mathcal{U}=\emptyset$.
\end{lister}
By the obvious induction hypothesis, $\dInd{K} (W^s\times P^{\,x_s})<\alpha$ for each~$s$. Finally, $P=(X\times
Y)\setminus\bigcup\mathcal{U}=\bigcup_{s\in S} (W^s\times P^{\,x_s})$ is a $K$-partition for $\mathcal{F}$, and $\dInd{K} P<\alpha$.
We have shown that $\dInd{K}(X\times Y)\leq\alpha=\dInd{K} Y$.

(b) In view of Proposition \ref{subspace}, we infer that if $\dstr{K}(X\times Y)=0$, then $\dstr{K} Y=0$. The converse becomes
justified when analysing the proof in the previous para\-graph, we moreover consider the implication $(\ast)$.

(c) Again $\dInd{L} (X\times Y)\geq\dInd{L} Y$. Write $\alpha=\dInd{L} Y$. If $\alpha=-1$, the equality (c) holds. Assume that $\alpha
\geq 0$. Consider the Hilbert cube $Q=[-1,2]^{\aleph_0}$ equipped with the metric
$\varrho((s_i)_{i=0}^\infty,(t_i)_{i=0}^\infty)=\sum_{i=0}^\infty 2^{-i}|s_i-t_i|$, and assume that $L\subset [0,1]^{\aleph_0}$. There
exists a neighbourhood $R\subset Q$ of $L$ with a map $r\colon R\to L$ such that $r(t)=t$ for $t\in L$. Let
$\varepsilon=\inf\{\varrho(s,t)\,$:\ $s\in L$, $t\in Q\setminus R\}$. Take an arbitrary closed set $F\subset X\times Y$ with a map
$f\colon F\to L$. Since $L$ is an ANR, there exists an open neighbourhood $U$ of $F$ with an extension $g\colon \cl U\to L$ of~$f$.
For each point $x\in X$, consider the open set $U^x=\pi_Y(U\cap\pi_X^{-1}(x))$, the closed set $G^x=\pi_Y(\cl U\cap\pi_X^{-1}(x))$,
and the map $g^x\colon G^x\to L$, $g^x(b)=g(x,b)$ for $b\in G^x$. In~$Y$ there is an $L$-partition $P^x$ for $g^x$ with $\dInd{L}
P^x<\alpha$ and with an extension $\psi^x\colon Y\setminus P^x\to L$ of~$g^x$. As $\pi_X$ is a closed map, $\pi_X(F\setminus (X\times
U^x))\not\ni x$ is closed in~$X$. There is a neighbourhood $N^x$ of $x$ with  $F\cap\pi_X^{-1}(\cl N^x)\subset X\times U^x$. Writing
as usually $(s_i)_{i=0}^\infty\pm (t_i)_{i=0}^\infty=(s_i\pm t_i)_{i=0}^\infty$, we set
\begin{align*}
d^x\colon \pi_X^{-1}(x)\cup (F\cap\pi_X^{-1}(\cl N^x)) & \to [-1,1]^{\aleph_0},\\
d^x(a,b) & = \left\{
         \begin{array}{cl}
         g(a,b)-g(x,b)& \mbox{for
         }(a,b)\in F\cap\pi_X^{-1}(\cl N^x),\\
         0 & \mbox{for } a= x, b\in Y.
         \end{array}
          \right.
\end{align*}
The function $d^x$ is correctly defined and continuous. Let $e^x\colon X\times Y\to [-1,1]^{\aleph_0}$ be an extension of $d^x$.
Consider the point $o=(0,0,\ldots)\in Q$, the open ball B$(o,\varepsilon)$, and the closed set $\pi_X[(X\times
Y)\setminus(e^x)^{-1}(\mathrm{B}(o,\varepsilon))]\not\ni x$. There is a neigh\-bour\-hood $V^x\subset N^x$ of~$x$ with
$\pi_X^{-1}(V^x)\subset (e^x)^{-1}(\mathrm{B}(o,\varepsilon))$. Again, we take a finite clopen refinement $\{W^s\colon s\in S\}$ of
$\{V^x\colon x\in X\}$, where the sets $W^s$ are pairwise disjoint. We fix points $x_s$ with $W^s\subset V^{x_s}$, and we obtain
$F\cap\pi_X^{-1}(W^s)\subset W^s\times U^{x_s}$. By the obvious induction hypothesis, $\dInd{L} (W^s\times P^{\,x_s})<\alpha$ for
each~$s$, and $\dInd{L} P<\alpha$ for $P=\bigcup_{s\in S} (W^s\times P^{\,x_s})$. There remains to observe that the map
\begin{align*}
\varphi\colon (X\times Y)\setminus P & \to L,\\
\varphi(a,b) & =r(\psi^{x_s}(b)+e^{x_s}(a,b)) \mbox{ \ \ \ for $a\in W_s$ and $b\in Y\setminus P^{x_s}$}
\end{align*}
is correctly defined and extends $f$. Indeed, $\psi^{x_s}(b)+e^{x_s}(a,b)\in R$ since $\psi^{x_s}(b)\in L\subset [0,1]^{\aleph_0}$ and
$e^{x_s}(a,b)\in [-1,1]^{\aleph_0}\cap\mathrm{B}(o,\varepsilon)$. If $(a,b)\in F\cap\pi_X^{-1}(W^s)$, then $b\in U^{x_s}\subset
G^{x_s}$, $\psi^{x_s}(b)+e^{x_s}(a,b)=g(x_s,b)+d^{x_s}(a,b)=g(a,b)\in L$ and $\varphi (a,b)=g(a,b)=f(a,b)$. Therefore, $P$ is an
$L$-partition for~$f$, and $\dInd{L}(X\times Y)\leq\alpha$.
\end{proof}

The foregoing proof also works in the case when $X$ is paracompact and $\dInd{K} Y$, $\dInd{L} Y$ are integers (we need a compact $Y$
and $\dim X=0$, of course).

\section{Spreading out compact spaces in a plank}\label{plank}

Any suitably chosen subspace of a product or a product itself is sometimes called a plank. We shall additionally compress one of the
product's faces into one of the factors.

{\em Suppose that\/ $X$ and\/ $Y$ are non-empty compact spaces.} We shall recall the definition of the space $Z(X,Y)$, and investigate
its properties (cf.~\cite{krz-10a}). To begin, write $\mathcal{S}_X$ for the family of all subsets of~$X$ that are either finite (so
$\emptyset\in\mathcal{S}_X$), or homeomorphic to $A_{\aleph_0}$. Let
$\mathfrak{m}\geq\max\{\aleph_0,(wX)^+,(wY)^+,\card\mathcal{S}_X\}$, where $wX$ and $wY$ denote the weights of $X$ and $Y\!$, and put
$M=A_\mathfrak{m}\times X\times Y$. Let $\pi_1\colon\! M\to N$ be the quotient map that compresses sets $\{(\mu,x,y)\in M\colon y\in
Y\}$ for all $x\in X$\linebreak into points---here $N$ is the compact quotient space.

Given any function $\varphi\colon A_\mathfrak{m}\setminus\{\mu\}\to\mathcal{S}_X$ such that $\card\varphi^{-1}(S)=\mathfrak{m}$ for
every $S\in\mathcal{S}_X$, we put
\begin{align*}
H(\alpha) & = \left\{
         \begin{array}{ll}
         \pi_1(\{\mu\}\times X\times Y) & \mbox{ for }\alpha=\mu,\\
         \pi_1(\{\alpha\}\times\varphi(\alpha)\times Y) & \mbox{ for
         }\alpha\neq\mu,\mbox{ and}
         \end{array}
          \right.\\
Z(X,Y) & = {\bigcup}_{\alpha\in A_\mathfrak{m}} H(\alpha).
\end{align*}
(We slightly change the notation originating in \cite{krz-10a}.)

\begin{prop}[{\cite[Section 1]{krz-10a}}]\label{general-zxy}
$Z(X,Y)$ is a compact space. Every component of $Z(X,Y)$ is homeomorphic to some component of $X$ or $Y$. If $X$ and $Y$ are Fr\'echet
spaces, then so is $Z(X,Y)$.\qed
\end{prop}

The following results from Theorem \ref{compon1}.

\begin{lemma}\label{ddim}
$\ddim{L} Z(X,Y)=\max\{\ddim{L} X,\ddim{L} Y\}$.\qed
\end{lemma}

Write $\pi_X\colon Z(X,Y)\to X$ and  $\pi_{A_\mathfrak{m}}\colon Z(X,Y)\to A_\mathfrak{m}$ for projections, i.e.\ the unique maps such
that $\pi_X(\pi_1 (\alpha,x,y))=x$ and $\pi_{A_\mathfrak{m}}(\pi_1 (\alpha,x,y))=\alpha$ for every
$(\alpha,x,y)\in\pi_1^{-1}(Z(X,Y))$. Note that $\pi_{A_\mathfrak{m}}^{-1}(\alpha)=H(\alpha)$ for $\alpha\in A_\mathfrak{m}$, the
restriction $\pi_X|H(\mu)$ is a homeomorphism onto $X$, and $H(\alpha)$ is homeo\-morph\-ic to $\varphi(\alpha)\times Y$ for every
$\alpha\neq\mu$. A base of neighbourhoods of a point $\pi_1(\mu,x,y)\in H(\mu)$ consists of sets of the form
$\pi_{A_\mathfrak{m}}^{-1}(A)\cap\pi_X^{-1}(U)$, where $\mu\in A\subset A_\mathfrak{m}$, the complement $A_\mathfrak{m}\setminus A$ is
finite, and $U\subset X$ is a neighbourhood of $x$.

The space $Z(X,Y)$ depends on the choice of $\mathfrak{m}$, but this is insignificant in the present paper. The dependence
on~$\varphi$ is superficial because another function $\psi$ with $\card\psi^{-1}(S)=\mathfrak{m}$ for $S\in\mathcal{S}_X$ would yield
a new space homeomorphic to the former $Z(X,Y)$. Indeed, there would be a function $\xi\colon A_\mathfrak{m}\setminus\{\mu\}\to
A_\mathfrak{m}\setminus\{\mu\}$ such that $\varphi=\psi\circ\xi$. The homeomorphism in question would have fixed points of the form
$\pi_1(\mu,x,y)$, and would carry
$$H(\alpha)\ni\pi_1(\alpha,x,y)\;\longmapsto\;\pi_1(\xi(\alpha),x,y)\in \pi_1(\{\xi(\alpha)\}\times\psi(\xi(\alpha))\times Y)$$
for every $\alpha\neq\mu$. In particular, when $\mu\in A\subset A_\mathfrak{m}$ and $\card (A_\mathfrak{m}\setminus A)<\mathfrak{m}$,
we can think that---roughly speaking---$\pi_{A_\mathfrak{m}}^{-1}(A)$ has the same properties as $Z(X,Y)$. On the other hand, given a
non-empty closed set $F\subset X$, we can consider the~function $\chi\colon A_\mathfrak{m}\setminus\{\mu\}\to\mathcal{S}_F$,
$\chi(\alpha)=F\cap\varphi (\alpha)$, and it turns out that $\pi_X^{-1}(F)$ has the form of a $Z(F,Y)\subset\pi_1(A_\mathfrak{m}\times
F\times Y)$.

The following statement is a simple modification (with the same proof) of~\cite[Lemma 1]{krz-10a}.

\begin{lemma}\label{new-projection}
If\/ $G\subset Z(X,Y)$ is a\/ $G_\delta$-set {\em (\hspace{-.15ex}}so, also if\/ $G$ is open\/{\em ),} then there is a set\/
$A\subset A_\mathfrak{m}$ such that $\mu\in A$, $\card (A_\mathfrak{m}\setminus A)\leq \max\{wX,\aleph_0\}$, and\\
\parbox{13cm}{$$\pi_{A_\mathfrak{m}}^{-1}(A)\cap \pi_X^{-1}(\pi_X(G\cap H(\mu)))\subset G.$$}\hfill $\Box$
\end{lemma}

\section{Compact spaces with $\ddim{L}<\dInd{L}$, where $L$ is an ANR}\label{ANR-section} \label{fourth}

We go on to investigate the behaviour of $\dInd{L}$ under the operation $Z(X,Y)$.

\begin{lemma}\label{ANR-leq}
$\dInd{L} Z(X,Y)\leq\max\{\dInd{L} X +1,\dInd{L} Y\}$.
\end{lemma}

\begin{proof}
Take a closed subset $F$ of $Z=Z(X,Y)$ and a map $f\colon F\to L$. Since $L$ is an ANR, there exists a neighbourhood $U$ of~$F$ with
an extension $g\colon U\to L$ of~$f$. The restriction $\pi_X|H(\mu)$ is a homeomorphism onto $X$, and hence, there are open subsets
$V_0,V_1$ of $X$ such that
$$\pi_X (F\cap H(\mu))\subset V_0\subset\cl V_0\subset V_1\subset\cl V_1\subset\pi_X(U\cap H(\mu)).$$
Observe that $F\setminus \pi_X^{-1}(V_0)$ and $\pi_X^{-1}(\cl V_1)\setminus U$ are closed subsets of $Z$, and none of them meets $H(\mu)$. Their images under $\pi_{A_\mathfrak{m}}$ do not contain
$\mu$, and being closed, are finite. Thus,
$$A=A_\mathfrak{m}\setminus [\pi_{A_\mathfrak{m}}(F\setminus \pi_X^{-1}(V_0))\cup\pi_{A_\mathfrak{m}}(\pi_X^{-1}(\cl V_1)\setminus U)]\ni\mu$$
is clopen in $A_\mathfrak{m}$. Moreover
$$F\cap\pi_{A_\mathfrak{m}}^{-1}(A)\subset\pi_X^{-1}(V_0) \mbox{ \ \ \ and \ \ \ } \pi_X^{-1}(\cl V_1)\cap\pi_{A_\mathfrak{m}}^{-1}(A)\subset U.$$

For each $S\in\mathcal{S}_X$, let $x^S\in S$ be the limit of $S$ whenever $S$ is infinite. Choose a point $l_0\in L$. For each
$\alpha\in A\setminus\{\mu\}$, we shall define an extension $g'_\alpha\colon H(\alpha)\to L$ of the restriction
$g|(\pi_X^{-1}(V_0)\cap H(\alpha))$. Consider $S=\varphi(\alpha)$. There are two cases. (1)~If $\varphi(\alpha)$ is finite or
$x^{\varphi(\alpha)}\in V_1$, then $V_1\cap S$ is clopen in $S$ and $W_\alpha=H(\alpha)\cap\pi_X^{-1}(V_1)$ is clopen in~$Z$. (2)~If
$x^{\varphi(\alpha)}\not\in V_1$, then $V_0\cap S$ is clopen in $S$, and we put $W_\alpha=H(\alpha)\cap\pi_X^{-1}(V_0)$. Since
$W_\alpha\subset U$ in both cases, we can set
$$g'_\alpha(z)  = \left\{
         \begin{array}{cl}
         g(z)& \mbox{for }z\in W_\alpha,\\
         l_0 & \mbox{for } z\in H(\alpha)\setminus W_\alpha.
         \end{array}
          \right.$$

$\dInd{L} H(\alpha)=\dInd{L} Y$ for $\alpha\neq\mu$ by Theorem \ref{zero-product}(c). If $\alpha\in A_\mathfrak{m}\setminus A$, then
in $H(\alpha)$ we take an $L$-partition $P_\alpha$ with $\dInd{L} P_\alpha<\dInd{L} Y$ for the restriction $f|(F\cap H(\alpha))$. This
means that $F\cap H(\alpha)\subset H(\alpha)\setminus P_\alpha$ and there is an extension $f'_\alpha\colon H(\alpha)\setminus
P_\alpha\to L$ of~$f|(F\cap H(\alpha))$.

Since $A_\mathfrak{m}\setminus A$ is finite, the set
\begin{equation}\label{1}
P=(H(\mu)\setminus \pi_X^{-1}(V_0))\cup{\bigcup}_{\alpha\in A_\mathfrak{m}\setminus A} P_\alpha
\end{equation}
is closed in $Z$ and $\dInd{L} P<\max\{\dInd{L} X+1,\dInd{L} Y\}$. It is an $L$-partition for $f$ because the function
$$f'(z)  = \left\{
         \begin{array}{rl}
         g(z)& \mbox{for }z\in \pi_{A_\mathfrak{m}}^{-1}(A)\cap\pi_X^{-1}(V_0),\\
         g'_\alpha (z) & \mbox{for } z\in H(\alpha), \mbox{ where } \mu\neq\alpha\in A, \mbox{ and}\\
         f'_\alpha (z) & \mbox{for } z\in H(\alpha)\setminus P_\alpha, \mbox{ where } \alpha\in A_\mathfrak{m}\setminus A,
         \end{array}
         \right. $$
is correctly defined on $Z\setminus P$, continuous, and extends $f$.
\end{proof}

\begin{lemma}\label{Ind-raise1}
Suppose that\/ $X$ is a non-empty, compact Fr\'echet space, $F\subset B\subset X$ are closed, and\/ $f\colon F \to L$ is a map that
does not extend to a map from~$B$ to~$L$. Let\/ $G=\pi_X^{-1} (F)\cap H(\mu)$ and\/ $g=f\circ (\pi_X|G)\colon G\to L$. If\/ $P$ is an
$L$-partition in $Z=Z(X,Y)$ for~$g$, then one of the following conditions is satisfied:
\begin{lister}
\item $B\cap\int\pi_X (P\cap H(\mu))\neq \emptyset$;
\item there is an\/ $\alpha\neq\mu$ such that\/ $\varphi(\alpha)\in\mathcal{S}_X$ is infinite and\/ $\pi_{A_\mathfrak{m}}^{-1}(\alpha)\cap\pi_X^{-1}(x^{\varphi(\alpha)})\subset P$, where\/
$x^{\varphi(\alpha)}\in B\cap\varphi(\alpha)$ is the limit point of\/ $\varphi(\alpha)$ {\em (}\hspace{-.2ex}and the intersection of
the point-inverses is homeomorphic to\/ $Y)$.
\end{lister}
\end{lemma}

\begin{proof}
We need Borsuk's homotopy extension theorem in the following formulation: {\em Suppose that $f_1,f_2\colon F\to L$ are homotopic maps
from a closed subspace $F$ of a compact space~$B$ into an ANR~$L$. Then $f_1$ has an extension from $X$ to $L$ iff $f_2$ has such an
extension} (cf.\ \cite[Lemma 1.9.7 and its proof]{eng-95}).

By West's Theorem \ref{west}, there exists a polyhedron $|K|$ with maps $\gamma_1\colon L\to |K|$, $\gamma_2\colon |K|\to L$ such that
$\gamma_2\circ \gamma_1\simeq \id_L$ (the composition is homotopic to the identity $\id_L$ on $L$) and $\gamma_1\circ \gamma_2\simeq
\id_{|K|}$. Evidently $f\simeq \gamma_2\circ \gamma_1\circ f$. It follows from the homotopy extension theorem that $\gamma_1\circ f$
does not extend to a map from $B$ to~$|K|$. Moreover, each $L$-partition in $Z$ for $g$ is a $|K|$-partition for $\gamma_1\circ
f\circ(\pi_X|G)$. Thus, we can {\em assume without loss of generality that $L=|K|$, and $f,g$ are maps into~$|K|$.}

Consider the closed $K$-cover $\mathcal{K}$ of $|K|$ (see the definition before Lemma \ref{extend}), and take an open swelling
$\mathcal{U}$ of $\mathcal{K}$ such that $\cl\mathcal{U}$ is a $K$-tuple of~$|K|$.

Take any $|K|$-partition $P\subset Z\setminus G$ for~$g$, and assume that the interior $\int\pi_X (P\cap H(\mu))$ does not meet~$B$.
Let $g'\colon Z\setminus P\to |K|$ be an extension of~$g$. Consider the open $K$-cover $\mathcal{V}=g'^{-1}(\mathcal{U})$ of
$Z\setminus P$. Remembering that $\pi_X|H(\mu)$ is a homeomorphism onto $X$, write $\mathcal{W}=\pi_X(\mathcal{V}|H(\mu))$ and note
that $\mathcal{W}$ is a $K$-neigh\-bour\-hood of $f^{-1}(\mathcal{K})$. In $B$ choose an open swelling~$\mathcal{H}$ of
$(\cl\mathcal{W})|B$. We have $B=\bigcup\mathcal{H}$ since $B\subset\cl\pi_X (H(\mu)\setminus P)=\bigcup\cl\mathcal{W}$. It follows
that $(\cl\mathcal{W})|B$ is not a $K$-tuple (in the other case, $\mathcal{H}$ would be a $K$-neigh\-bour\-hood of
$f^{-1}(\mathcal{K})$, and $f$~would have an extension from $B$ to $|K|$ by Lemma \ref{extend}). Therefore, there is an $x_0\in
B\cap\bigcap_{i\in I}\cl W_i$, where $I\subset\{0,\ldots,k\}$ and the simplex with vertices $e_i$, $i\in I$, does not belong to~$K$.
For each $i\in I$, take a sequence $S_i\subset W_i$ converging to $x_0$ ($X$ is Fr\'echet), and put $S=\{x_0\}\cup\bigcup_{i=0}^k
S_i$. By Lemma \ref{new-projection}, there is a set $A\subset A_\mathfrak{m}$ with $\mu\in A$, $\card (A_\mathfrak{m}\setminus
A)<\mathfrak{m}$, and $\pi_{A_\mathfrak{m}}^{-1}(A)\cap \pi_X^{-1}(W_i)\subset V_i$ for each~$i\in I$. As
$\card\varphi^{-1}(S)=\mathfrak{m}$, we can find an $\alpha\in A\setminus\{\mu\}$ such that $\varphi(\alpha)=S$.

If $i\in I$, then every point $\pi_1(\alpha,x_0,y)\in \pi_1(\{\alpha\}\times\{x_0\}\times Y)$ is the limit of the sequence
$\pi_1(\{\alpha\}\times S_i\times\{y\})\subset V_i=g'^{-1}(U_i)$. If we had $\pi_1(\alpha,x_0,y)\not\in P$, then we would obtain
$g'(\pi_1(\alpha,x_0,y))\in\cl U_i$ for $i\in I$, and $\bigcap_{i\in I}\cl U_i$ would be non-empty. As $\cl\mathcal{U}$ is a
$K$-tuple, we infer that $\pi_{A_\mathfrak{m}}^{-1}(\alpha)\cap\pi_X^{-1}(x_0)=\pi_1(\{\alpha\}\times\{x_0\}\times Y)\subset P$.
Finally, we can write $x^{\varphi(\alpha)}=x_0$.
\end{proof}

Let $X$ be a normal space and $b\in X$. Bearing in mind the convention that $\infty$ is bigger than any ordinal, we define
$$\dInd{L}_{b+}X=\min\label{indb+}\{\alpha\colon \mbox{there is a neighborhood $U$ of $b$ with $\dInd{L} \cl U\leq \alpha\}$}.$$
Note that if $B\subset X$ is closed and $b\in B$, then $\dInd{L}_{b+}B \leq\dInd{L}_{b+}X\leq \dInd{L} X $.

\begin{lemma}\label{Ind-raise2} Suppose that $X$ is a non-empty, compact Fr\'echet space, and $B$ is a closed subspace of~$X$. Let
$z\in H(\mu)$ be any point such that $c=\pi_X(z)\in B$ and $\dInd{L}_{c+} B\geq 1$. If $\dInd{L}_{b+} X\geq\alpha$ for each $b\in B$,
then
$$\dInd{L}_{z+} Z(X,Y)\geq\min\{\alpha,\dInd{L} Y\}+1.$$
\end{lemma}

\begin{proof}
It suffices to show that $\dInd{L} (\pi_{A_\mathfrak{m}}^{-1}(A)\cap\pi_X^{-1}(\cl U))\geq \min\{\alpha,\dInd{L} Y\}+1$ for any base
neighbourhood $\pi_{A_\mathfrak{m}}^{-1}(A)\cap\pi_X^{-1}(U)$ of~$z$, where $\mu\in A\subset A_\mathfrak{m}$, $A_\mathfrak{m}\setminus
A$ is finite, and $U$ is a neigh\-bour\-hood of~$c$. Let $V\ni c$ be open in $X$ and $\cl V\subset U$. We have $\dInd{L} (B\cap\cl
V)\geq 1$ as $\dInd{L}_{c+} B\geq 1$, and there is a closed set $F\subset B\cap \cl V$ with a map $f\colon F\to L$ that does not have
an extension from $B\cap\cl V$ to~$L$. Let $G=\pi_X^{-1}(F)\cap H(\mu)$ and $g=f\circ (\pi_X|G)$. Take an arbitrary $L$-partition $P$
for~$g$ in $\pi_{A_\mathfrak{m}}^{-1}(A)\cap\pi_X^{-1}(\cl U)$, which has the form of $Z(\cl U,Y)\subset\pi_1(A\times \cl U\times Y)$.
By Lemma \ref{Ind-raise1}, two cases may arise. (1)~Some $b\in B\cap\cl V$ is an interior point of $\pi_X (P\cap H(\mu))$ in~$\cl U$.
Then there is a neighbourhood $W\subset U\cap\pi_X (P\cap H(\mu))$ of $b$ in~$X$. In consequence, $\dInd{L}
P\geq\dInd{L}(\pi^{-1}_X(\cl W)\cap H(\mu))=\dInd{L}\cl W\geq\alpha$ because $\dInd{L}_{b+} X\geq\alpha$. (2)~$P$~contains a
homeomorphic copy of~$Y$, and then $\dInd{L} P\geq\dInd{L} Y$. Thus, $\dInd{L} P$~$\geq$ $\min\{\alpha,\dInd{L} Y\}$ in both cases,
which proves the lemma.
\end{proof}

For any normal space $X$, let us write
$$K(X)=\{b\in X\colon \dInd{L}_{b+} X=\dInd{L} X\}.$$
Observe that $K(X)$ is a closed subset of~$X$.

\begin{theo}\label{Ind-raise3}
Suppose that\/ $X$ and\/ $Y$ are non-empty compact spaces, and\/ $X$ is Fr\'echet. If\/ $\dInd{L} X=\dInd{L} Y$ and\/ $\dInd{L}
K(X)\geq 1$, then
$$\dInd{L} Z(X,Y)=\dInd{L} X+1.$$
\end{theo}

\begin{proof}
The inequality ``$\leq$'' results from Lemma \ref{ANR-leq}.

Assume that $\dInd{L} K(X)\geq 1$. The equality $\dInd{L}=0$ is equivalent to $\ddim{L}=0$. We claim that {\em there is a point $c\in
K(X)$ with $\dInd{L}_{c+} K(X)\geq 1$.} In the other case, using the compactness of $K(X)$, we could cover $K(X)$ by sets
$U_1,\ldots,U_n$ open in $K(X)$ and such that $\dInd{L} \cl U_i=0$ for $i=1,\ldots,n$. By the countable sum theorem for $\ddim{L}$
(Fe\-dor\-chuk \cite[Proposition 5.1]{fed-09}), we would obtain $\ddim{L} K(X)=0=\dInd{L} K(X)$. A contradiction. Therefore, we can
put $B=K(X)$ and apply Lemma \ref{Ind-raise2}.
\end{proof}

\begin{lemma}\label{kc}
If $X$ is a separable metric space with $\dInd{L} X=n\in\mathbb{N}$, then $K(X)$ is non-empty, and $\dInd{L}_{b+} K(X)=n$ for each
$b\in K(X)$.
\end{lemma}

\begin{proof}
Theorem \ref{gen-relations} implies that $\ddim{L}=\dInd{L}$ for closed subspaces of~$X$. $X$~has a countable base $\mathcal{B}$, and
$X\setminus K(X)$ is the union of a sequence $\cl U_i$, where $U_i\in\mathcal{B}$ and $\dInd{L}\cl U_i<n$ for $i=0,1,\ldots\mbox{}$ If
we had $\dInd{L} K(X)<n$, then we would obtain $\dInd{L} X<n$ by the countable sum theorem for $\ddim{L}$ (Fedorchuk \cite[Proposition
5.1]{fed-09}). Thus $\dInd{L} K(X)=n$.

Let $b\in K(X)$, and $U$ be a neighbourhood of $b$ in $K(X)$. Using the hereditary normality of~$X$, one can find a neighbourhood $V$
of~$b$ in~$X$ such that $U=V\cap K(X)$ and $\cl U=\cl V\cap K(X)$. Then $\dInd{L}\cl V=n$. By the same argument as in the first
paragraph, we infer that $\dInd{L}\cl U=n$. Therefore $\dInd{L}_{b+} K(X)\geq n$.
\end{proof}

\begin{theo}\label{ANR-main}
Let $L$ be a compact metric ANR. Suppose that $C$ is a metric continuum with $1\leq n=\ddim{L} C<\infty$. For each ordinal $\alpha\geq
n$, there exists a compact Fr\'echet space $X_{C,\alpha}$ such that
\begin{lister}
\item $\ddim{L} X_{C,\alpha} =n$,
\item $\dInd{L} X_{C,\alpha} =\alpha$, and
\item each component of $X_{C,\alpha}$ is homeomorphic to $C$.
\end{lister}
\end{theo}

\begin{proof}
$K(C)$ is closed in $C$, and $n=\dInd{L}_{b+} K(C)\leq\dInd{L}_{b+} C\leq n$ for each $b\in K(C)$ (Lemma \ref{kc}). By transfinite
induction on $\alpha$, we shall construct compact Fr\'echet spaces $X_{C,\alpha}$, $\alpha\geq n$, and closed subspaces
$B_\alpha\subset X_{C,\alpha}$ such that
\begin{lister}
\item every component of $X_{C,\alpha}$ is homeomorphic to~$C$;
\item $B_\alpha$ is homeomorphic to~$K(C)$;
\item $\dInd{L} X_{C,\alpha}\leq\alpha$; and
\item $\dInd{L}_{b+} X_{C,\alpha}\geq\alpha$ for each $b\in B_\alpha$.
\end{lister}

For $\alpha=n$, let $X_{n,n}=C$ and $B_n=K(C)$. Assume $X_{C,\alpha}\supset B_\alpha$ are compact, Fr\'echet, and satisfy (a--d). Let
$X=Y=X_{C,\alpha}$, $\mathfrak{m}=\max\{(wX_{C,\alpha})^+,\card\mathcal{S}_{X_{C,\alpha}}\}$,
$X_{C,\alpha+1}=Z(X_{C,\alpha},X_{C,\alpha})$, and $B_{\alpha+1}=H(\mu)\cap\pi_X^{-1}(B_\alpha)$. By Proposition \ref{general-zxy},
$X_{C,\alpha+1}$ is Fr\'echet, and each of its components is homeomorphic to~$C$. The restriction $\pi_X|B_{\alpha+1}$ is a
homeomorphism onto $B_\alpha$. $\dInd{L}X_{C,\alpha+1}\leq\alpha+1$ by Lemma \ref{ANR-leq}, and $\dInd{L}_{b+}
X_{C,\alpha+1}\geq\alpha+1$ for each $b\in B_{\alpha +1}$ by Lemma \ref{Ind-raise2}.

Assume that $\alpha$ is a limit ordinal, and there are $X_{C,\beta}\supset B_\beta$ for $\beta<\alpha$. Let~$D$ be the one-point
compactification of the topological sum ${\bigoplus}_{\beta<\alpha} X_{C,\beta}$, and $d_0\in D$ the unique point in the remainder. In
the disjoint sum of $C$ and $D$, identify~$d_0$ with a point $c_0\in C$, and call the resulting space $Y$. Using the fact that
$A_\mathfrak{n}$ is Fr\'echet for every $\mathfrak{n}$, one routinely checks that $Y$ is Fr\'echet. It follows from Lemma
\ref{zero-Dowker} that $\dInd{L} Y=\alpha$. Put $X=C$, $\mathfrak{m}=2^{\aleph_0} +(\sup\{wX_{C,\beta}\colon\beta<\alpha\})^+$,
$X_{C,\alpha}=Z(X,Y)$, and $B_\alpha=\pi_X^{-1}(K(C))\cap H(\mu)$. $X_{C,\alpha}$ is Fr\'echet, (a, c) are satisfied (see Proposition
\ref{general-zxy} and Lemma~\ref{ANR-leq}), and (b, d) are evident.

The conditions (c, d) yield the equality $\dInd{L}X_{C,\alpha}=\alpha$, and $\ddim{L} X_{C,\alpha}=n$ by Theorem \ref{compon1}.
\end{proof}

\begin{rem} \em
The construction in the above proof is essentially the same as in the proof of Theorem 5 in \cite{krz-10a} (see Remarks 3--4 therein),
which yields a compact Fr\'echet space $X_{C,\alpha}$ with $\dim X_{C,\alpha}=n$, $\trind X_{C,\alpha}=\trInd X_{C,\alpha}=\alpha$,
and with components homeomorphic to~$C$. The proofs of Lemmas \ref{ANR-leq} and \ref{Ind-raise2} in the present\linebreak paper are
more complex than the proofs of corresponding Lemmas 6 and 7 in~\cite{krz-10a}.

We may add at this place that Lemma 7 in \cite{krz-10a} needs one more assumption (necessary but missed out): the space $B$ in that
statement should be a non-degenerate\linebreak continuum (then each component of a non-empty open subspace is uncountable).
\end{rem}

Proposition \ref{cube} and Theorem \ref{ANR-main} yield

\begin{cor}\label{ANR-cor}
Let $L$ be a non-contractible, compact metric ANR,\/ $1\leq n\in\mathbb{N}$, and let $\alpha\geq n$ be an ordinal. If\/ $n=1$ or the
join $L\ast L$ is non-contractible, then there exists a compact Fr\'echet space $X_{n,\alpha}$ such that
\begin{lister}
\item $\ddim{L} X_{n,\alpha}=n$,
\item $\dInd{L} X_{n,\alpha}=\alpha$, and
\item each component of $X_{n,\alpha}$ is homeomorphic to a cube $[0,1]^{m}$ for a certain natural number $m=m(L,n)$.\qed
\end{lister}
\end{cor}

\section{Compact spaces with $\ddim{K}<\dInd{K}$ or $\dInd{K}<\dInd{|K|}$, where $K$~is a simplicial complex} \label{fifth}

This section is devoted to the behaviour of $\dInd{K}$ under the operation $Z(X,Y)$. We obtain inequalities for $\dInd{K}$ that
resemble those in preceding section for $\dInd{L}$, and we establish conditions in order that $\dInd{K} Z(X, X) = \dInd{K} X$ or
$\dInd{K} Z(X, X) = \dInd{K} X+ 1$.

\begin{lemma}\label{proj-tuple}
If $\mathcal{F}$ is a closed $K$-tuple of $Z(X,Y)$, then there is a set $A\subset A_\mathfrak{m}$ such that $\mu\in A$, $A_\mathfrak{m}\setminus A$ is finite, and
$\pi_X(\mathcal{F}|\pi^{-1}_{A_\mathfrak{m}}(A))$ is a closed $K$-tuple of $X$.
\end{lemma}

\begin{proof} Take any closed $K$-tuple $\mathcal{F}$ of $Z(X,Y)$. Then, the $K$-tuple $\pi_X(\mathcal{F}|H(\mu))$ has a $K$-neigh\-bour\-hood $\mathcal{U}$ in $X$. Since $\pi_X(F_i\cap H(\mu))\subset
U_i$ for $i=0,\ldots,k$, we have $\mu\not\in A_i=\pi_{A_\mathfrak{m}}(F_i\setminus\pi_X^{-1}(U_i))$ for each $i$. Since $A_i$ are closed in $A_\mathfrak{m}$, they are finite. As easily checked,
$A=A_\mathfrak{m}\setminus\bigcup_{i=0}^k A_i$ has the required properties.
\end{proof}

\begin{lemma}\label{big-half}
Suppose that $\mathcal{U}$ is an open $K$-tuple of $X$, and $\dobs{K}\mathcal{U}=\emptyset$. Then there is a $K$-neighbourhood
$\mathcal{V}$ of $\pi_X^{-1}(\mathcal{U})$ in $Z(X,Y)$ with
$$Z(X,Y)\setminus\bigcup\mathcal{V}=H(\mu)\setminus\pi_X^{-1}\left(\bigcup\mathcal{U}\right).$$

If moreover $K=\partial\Delta^k$, where $k\geq 2$, and\/ $\cl\mathcal{U}$ is a $\partial\Delta^k$-tuple, then
$\mathcal{V}$ can be chosen so that $\dobs{\partial\Delta^k}\mathcal{V}=\emptyset$.
\end{lemma}

\begin{proof}
Each $S\in\mathcal{S}_X$ is metrisable, and by Lemma \ref{smaller-partition}, the $K$-tuple $\mathcal{U}|S=(U_0\cap S,\ldots,U_k\cap
S)$ has a $K$-neighbourhood $\mathcal{V}^S$ in~$S$ which covers $S$ (a direct proof is easy, too). Let $\alpha\neq\mu$. Then $\pi_X$
maps $H(\alpha)$ onto $S=\varphi(\alpha)$. The sets $\pi_X^{-1}(V^{\varphi(\alpha)}_i)\cap\linebreak H(\alpha)$, $i=0,\ldots,k$, form
an open $K$-cover of $H(\alpha)$. Now, the unions
$$V_i=\pi_X^{-1}(U_i)\cup{\bigcup}_{\alpha\in A_\mathfrak{m}\setminus \{\mu\}}
(\pi_X^{-1}(V^{\varphi(\alpha)}_i)\cap H(\alpha))$$ form the requested $K$-neighbourhood $\mathcal{V}$
of~$\mathcal{F}$.

Assume that $k\geq 2$, and $\cl\mathcal{U}$ is a $\partial\Delta^k$-tuple. Then there is a $\partial\Delta^k$-neighbourhood
$\mathcal{W}$ of $\cl\mathcal{U}$. Take an $S\in\mathcal{S}_X$, and let $x^S\in S$ be the limit of~$S$ if $S$ is infinite. We choose
an index $i^S\in\{0,\ldots,k\}$ so that (1) $i^S=0$ when $S$ is finite or $x^S\not\in\bigcup\mathcal{W}$, and (2) $x^S\in W_{i^S}$
when $x^S\in\bigcup\mathcal{W}$. Now, we define a $\partial\Delta^k$-cover $\mathcal{W}^S$ of $X$ by the formulas
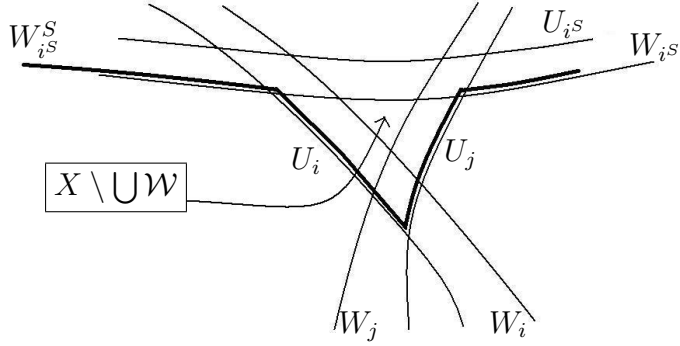
\begin{figure}
\begin{picture}(0,129)(12,3)
\put(-118,0){\includegraphics[scale=.39]{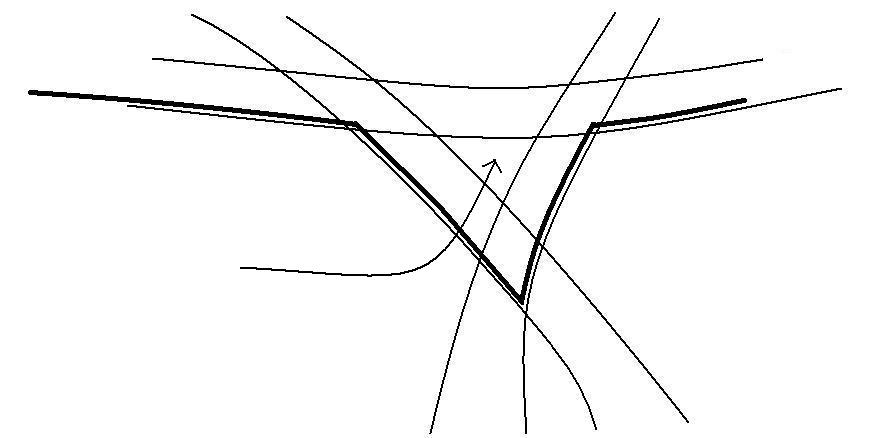}}
\put(-115,108){$W_{i^S}^S$}\put(120,105){$W_{i^S}$}\put(87,114){$U_{i^S}$}\put(10,0){$W_j$}\put(50,65){$U_j$}\put(-8,62){$U_i$}\put(67,0){$W_i$}
\put(-101,52){\framebox{$X\setminus\bigcup\mathcal{W}$}}
\end{picture}
\caption{\label{fig}Check that $\mathcal{W}^S$ and $\mathcal{N}^x$ are $\partial\Delta^k$-tuples ($i$, $j$, $i^S$ above are distinct---this is why we need $k\geq 2$).}\end{figure}
$$W^S_i =
\left\{\begin{array}{ll}
W_{i^S}\cup (X\setminus\bigcup\mathcal{U}) & \mbox{ for } i=i^S,\\
U_i &  \mbox{ for } i\neq i^S
\end{array}\right.$$
(see Figure \ref{fig}; in general, $W_{i^S}^S$ is not open!), and we put $\mathcal{V}^S=\mathcal{W}^S|S$. Since $x^S$ is the unique non-isolated point of an infinite $S$, it is easily seen that
$V^S_{i^S}=S\cap W_{i^S}^S$ is open in $S$. Hence, $\mathcal{V}^S$ is a $\partial\Delta^k$-neighbourhood of $\mathcal{U}|S$. We define $V_i$'s and $\mathcal{V}$ by the same formula as in the first
paragraph of this proof.

There remains to prove that $\mathcal{V}$ has an empty $\dobs{\partial\Delta^k}\mathcal{V}$. If $z\in Z(X,Y)\setminus\bigcup\mathcal{V}\subset H(\mu)$, then $x=\pi_X(z)\in
X\setminus\bigcup\mathcal{U}$. There are two cases. (A)~When $x\in W_{i^x}$ for some index $i^x$, we put $N^x=W_{i^x}$. (B)~When $x\not\in\bigcup\mathcal{W}$, we put
$N^x=X\setminus\bigcup\cl\mathcal{U}$ and $i^x=0$. Thus, $N^x$ is an open neighbourhood of~$x$, and the sets
$$V_i' =
\left\{\begin{array}{ll}
V_i\cup \pi_X^{-1}(N^x) & \mbox{ for } i=i^x,\\
V_i &  \mbox{ for } i\neq i^x
\end{array}\right.$$
are open in $Z(X,Y)$. We are to show that their intersection is empty. In order to check that $H(\mu)\cap\bigcap_{\,i=0}^{\,k} V_i'=\emptyset$, observe that $\pi_X(H(\mu)\cap V_i')$ is either
$U_{i^x}\cup N^x$ for $i=i^x$ or $U_i$ for $i\neq i^x$. These $k+1$ subsets of $X$ do not intersect in both cases (A) and (B), and we are done. When $\alpha\ne\mu$ and $S=\varphi(\alpha)$, we have
$H(\alpha)\cap V_i'=$ $H(\alpha)\cap\pi_X^{-1}(N_i^x)$, where
$$N^x_i =
\left\{\begin{array}{ll}
W_{i^x}^S\cup N^x & \mbox{ for } i=i^x, \\
W_i^S &  \mbox{ for } i\neq i^x
\end{array}\right\}=\left\{\begin{array}{ll}
U_i\cup N^x & \mbox{ if } i=i^x\neq i^S,\\
W_{i^S}\cup (X\setminus\bigcup\mathcal{U}) & \mbox{ if } i=i^S,\\
U_i &  \mbox{ for } i\not\in\{i^S,i^x\}.
\end{array}\right.$$
One checks that $\mathcal{N}^x=(N_0^x,\ldots,N_k^x)$ is a $\partial\Delta^k$-tuple in both cases (A) and~(B), and hence
$H(\alpha)\cap\bigcap_{\,i=0}^{\,k} V_i'=\emptyset$. Now, we infer that the sets $V_i'$ form a $\partial\Delta^k$-neigh\-bour\-hood
of~$\mathcal{V}$. Finally, $z\not\in\dobs{\partial\Delta^k}\mathcal{V}$ because $z\in V'_{i^x}\subset\bigcup_{i=0}^k V_i'$.
\end{proof}

As $Z(X,Y)$ contains homeomorphic copies of both $X$ and $Y$, we immediately obtain the inequality $\max\{\dInd{K} X,\dInd{K}
Y\}\leq\dInd{K} Z(X,Y)$. The following theorem contains upper bounds of $\dInd{K} Z(X,Y)$.

\begin{theo}\label{Indleq1}
Suppose that\/ $X$ and\/ $Y$ are non-empty compact spaces. Then $$\dInd{K} Z(X,Y)\leq\max\{\dInd{K} X+1,\dInd{K} Y\}.$$

If moreover $\dstr{K} X=0$, then $$\dInd{K} Z(X,Y)=\max\{\dInd{K} X,\dInd{K} Y\}.$$

If $k\geq 2$ and $\dInd{\partial\Delta^k}Y< \dInd{\partial\Delta^k} X+1=\dInd{\partial\Delta^k} Z(X,Y)$ then $$\dstr{\partial\Delta^k} Z(X,Y)=0.$$
\end{theo}

\begin{proof}
Take a closed $K$-tuple $\mathcal{F}$ of $Z=Z(X,Y)$. Lemma \ref{proj-tuple} yields a set $A\subset A_\mathfrak{m}$ such that $\mu\in
A$, $A_\mathfrak{m}\setminus A$ is finite, and $\pi_X(\mathcal{F}|\pi^{-1}_{A_\mathfrak{m}}(A))$ is a closed $K$-tuple of~$X$. Then
there is a $K$-neighbourhood $\mathcal{U}$ of $\pi_X(\mathcal{F}|\pi^{-1}_{A_\mathfrak{m}}(A))$ such that $\cl\mathcal{U}$ is a
$K$-tuple. Clearly $\dobs{K}\mathcal{U}=\emptyset$, and writing $P=X\setminus\bigcup\mathcal{U}$, we obtain $\dInd{K} P\leq\dInd{K}
X$. As $A_\mathfrak{m}\setminus A$ is finite, we can think that $\pi^{-1}_{A_\mathfrak{m}}(A)$ is a $Z(X,Y)$. Hence by
Lemma~\ref{big-half}, $\pi_X^{-1}(\mathcal{U})|\pi_{A_\mathfrak{m}}^{-1}(A)$ has a $K$-neighbourhood $\mathcal{V}$ in
$\pi_{A_\mathfrak{m}}^{-1}(A)$ with the corresponding $K$-partition $Q=H(\mu)\setminus\pi_X^{-1}(\bigcup\mathcal{U})$. Thus, $Q$~is a
$K$-partition in $\pi^{-1}_{A_\mathfrak{m}}(A)$ for $\mathcal{F}|\pi^{-1}_{A_\mathfrak{m}}(A)$. As $\pi_X|Q$ is a homeomorphism onto
$P$, we have $\dInd{K} Q\leq\dInd{K} X$. On the other hand, $H(\alpha)$ is homeomorphic to $\varphi(\alpha)\times Y$ for
$\alpha\neq\mu$, and $\dInd{K} H(\alpha)=\dInd{K} Y$ by Theorem \ref{zero-product}(a). For each $\alpha\not\in A$, in
$H(\alpha)=\pi_{A_\mathfrak{m}}^{-1}(\alpha)$ there is a $K$-neigh\-bour\-hood $\mathcal{W}^{\,\alpha}$ of
$\mathcal{F}|\pi_{A_\mathfrak{m}}^{-1}(\alpha)$ such that
$R^{\,\alpha}=\pi_{A_\mathfrak{m}}^{-1}(\alpha)\setminus\bigcup\mathcal{W}^{\,\alpha}$ has $\dInd{K}R^{\,\alpha}<\dInd{K}Y$. Since
$A_\mathfrak{m}\setminus A$ is finite, the union
\begin{equation}\label{2}
R=\left( H(\mu)\setminus\pi_X^{-1}\left(\bigcup\mathcal{U}\right)\right)\cup{\bigcup}_{\alpha\in A_{\mathfrak{m}}\setminus A}
R^{\,\alpha}
\end{equation}
is a $K$-partition for $\mathcal{F}$, and $\dInd{K} R<\max\{\dInd{K} X+1,\dInd{K} Y\}$. We have shown the first inequality of the
theorem's assertion.

In the case when $\dstr{K} X=0$, only a slight modification of the above proof is needed. Indeed, we do not need the $K$-tuple
$\cl\mathcal{U}$, but instead, $\pi_X(\mathcal{F}|\pi^{-1}_{A_\mathfrak{m}}(A))$ has a $K$-neighbourhood $\mathcal{U}$ such that
$\dobs{K}\mathcal{U}=\emptyset$ and the corresponding $K$-partition\linebreak $P$ satisfies the inequality $\dInd{K}P<\dInd{K} X$. At
the end, we obtain $\dInd{K} R<\max\{\dInd{K} X,\dInd{K} Y\}$ and $\dInd{K} Z\leq\max\{\dInd{K} X,\dInd{K} Y\}$.

If $k\geq 2$, $K=\partial\Delta^k$, and $\dInd{\partial\Delta^k}Y< \dInd{\partial\Delta^k} X+1=\dInd{\partial\Delta^k} Z$, then we
make another modification. We take $\mathcal{U}$ with the $\partial\Delta^k$-tuple $\cl\mathcal{U}$, and Lemma \ref{big-half} yields
$\mathcal{V}$ with $\dobs{\partial\Delta^k}\mathcal{V}=\emptyset$. As $\dInd{\partial\Delta^k}Y< \dInd{\partial\Delta^k} X+1$, for
$\alpha\not\in A$ we can take any $\partial\Delta^k$-neighbourhood $\mathcal{W}^{\,\alpha}$ in $\pi_{A_\mathfrak{m}}^{-1}(\alpha)$ of
$\mathcal{F}|\pi_{A_\mathfrak{m}}^{-1}(\alpha)$ with $\cl\mathcal{W}^{\,\alpha}$ being a $\partial\Delta^k$-tuple, in addition. Then
$\dobs{\partial\Delta^k} \mathcal{W}^{\,\alpha}=\emptyset$ and $\dInd{\partial\Delta^k} R<\dInd{\partial\Delta^k} X+1$. $R$~is the
corresponding $\partial\Delta^k$-partition of the open $\partial\Delta^k$-tuple which consists of the sets $V_i\cup\bigcup_{\alpha\in
A_{\mathfrak{m}}\setminus A}W^{\alpha}_i$ for $i=0,\ldots,k$, and which does not have $\partial\Delta^k$-obstruction points. This
completes the proof of the equality $\dstr{\partial\Delta^k} Z=0$.
\end{proof}

Proposition \ref{big-space2} and Theorem \ref{Indleq1} yield

\begin{cor}\label{weak}
Let $k\geq 2$ and $n\geq 1$. If $X$ is a compact metric space such that\/ $\dim X=kn$, then
$$\dInd{\partial\Delta^k} Z(X,Y)=\max\{n,\dInd{\partial\Delta^k} Y\}$$
for every non-empty compact space $Y$.\qed
\end{cor}

\begin{cor}\label{example2}
Let $k\geq 2$ and $n\geq 1$. If\/ $C$ is a metric continuum with $\dim C=kn$, then $X_C=Z(C,C)$ is a compact Fr\'echet space such that
\begin{lister}
\item $\ddim{\partial\Delta^k} X_C=\dInd{\partial\Delta^k} X_C=n$,
\item $\dInd{|\partial\Delta^k|} X_C=n+1$, and
\item each component of $X_C$ is homeomorphic to $C$.
\end{lister}
\end{cor}

\begin{proof}
It follows from Proposition \ref{general-zxy} that $X_C$ is a compact Fr\'echet space that satisfies the statement~(c).

All four of the dimensions $\ddim{\partial\Delta^k}$, $\ddim{|\partial\Delta^k|}$, $\dInd{\partial\Delta^k}$, and
$\dInd{|\partial\Delta^k|}$ of $C$ are equal to $n$ by Theorems \ref{gen-relations} and \ref{equiv}. Now, the statements
\ref{gen-relations}, \ref{ddim}, and \ref{weak} imply (a). The statement (b) results from \ref{Ind-raise3} and \ref{kc}.
\end{proof}

Since any simplicial complex $K$ is a triangulation of the polyhedron $|K|$, we may restate Fedorchuk's \cite[Question 3.1]{fed-10b}
as follows: {\em Are the dimensions $\dInd{K}$ and\/ $\dInd{|K|}$ equal for arbitrary normal spaces?} The foregoing corollary shows
that the answer is no. In the simplest case---for $k=2$, $n=1$, and $[0,1]^2$---we obtain $\dInd{\partial\Delta^2}
Z([0,1]^2,[0,1]^2)=1<2=\dInd{|\partial\Delta^2|} Z([0,1]^2,[0,1]^2)$.

The two above corollaries show that if we take a $kn$-dimensional compact metric space, then one-time use of the operation $Z(X,Y)$
does not allow us to obtain a space with $\ddim{\partial\Delta^k}<\dInd{\partial\Delta^k}$. We could try to iterate the operation.
However, {\em we even do not know whether $\dstr{\partial\Delta^2} Z([0,1]^2,[0,1]^2)$ is\/ $1$ or it is\/ $0$.} Let us write
$T=Z([0,1]^2,[0,1]^2)$. The values of $\dInd{\partial\Delta^2} Z(T,T)$ and $\dInd{\partial\Delta^2}Z(T,[0,1]^2)$ remain unknown. On
the other hand, $\dInd{\partial\Delta^2} Z([0,1]^2, T)=1$.

To show that the operation $Z(X,Y)$ sometimes raises the dimension $\dInd{K}$ by one, we need the following.

\begin{lemma} \label{big-partition2}
Suppose that\/ $X$ is a compact Fr\'echet space with\/ $\dInd{\partial\Delta^k} X=\alpha$ and\/ $\dstr{\partial\Delta^k} X=1$. Let\/
$\mathcal{F}$ be a\/ $\partial\Delta^k$-tuple in\/ $X$, where\/ $k\geq 1$. Assume that if\/ $\mathcal{U}$~is a\/
$\partial\Delta^k$-neighbourhood of\/ $\mathcal{F}$, and the corresponding\/ $\partial\Delta^k$-partition\/
$P=X\setminus\bigcup\mathcal{U}$ has\/ $\dInd{\partial\Delta^k} P<\alpha$, then\/ $\dobs{\partial\Delta^k}\mathcal{U}\neq\emptyset$.
Write $\mathcal{G}=\pi_X^{-1}(\mathcal{F})|H(\mu)$. If\/ $Q$~is a\/ $\partial\Delta^k$-partition in\/ $Z(X,Y)$ for\/~$\mathcal{G}$,
then one of the following conditions is satisfied:
\begin{lister}
\item $\dInd{\partial\Delta^k}(Q\cap H(\mu))=\alpha$;
\item there is an\/ $\alpha\neq\mu$ such that\/ $\varphi(\alpha)\in\mathcal{S}_X$ is infinite and\/
$\pi_{A_\mathfrak{m}}^{-1}(\alpha)\cap\pi_X^{-1}(x^{\varphi(\alpha)})\subset Q$, where\/ $x^{\varphi(\alpha)}$ is the accumulation
point of\/ $\varphi(\alpha)$ {\em (}\hspace{-.2ex}and the intersection of the point-inverses is homeomorphic to\/ $Y)$.
\end{lister}
\end{lemma}

\begin{proof}
Let $\mathcal{V}$ be any $\partial\Delta^k$-neighbourhood of~$\mathcal{G}$ in $Z=Z(X,Y)$, and $Q$ the corresponding
$\partial\Delta^k$-partition. Since $\pi_X|H(\mu)$ is a homeomorphism onto~$X$, assume that $\dInd{\partial\Delta^k}(Q\cap
H(\mu))<\alpha$. Hence, $\mathcal{U}=\pi_X(\mathcal{V}|H(\mu))$ has $\emptyset\neq\dobs{\partial\Delta^k}\mathcal{U}$. By Lemma
\ref{obs}, there is a common element $x_0\in\cl({\bigcap}_{\,0\leq j\leq k,\,j\neq i} U_j)$ for $i=0,\ldots,k$. As $X$ is Fr\'echet,
for each $i$ there is an infinite sequence $S_i\subset {\bigcap}_{\,0\leq j\leq k,\,j\neq i} U_j$ that converges to $x_0$. It follows
from Lemma~\ref{new-projection} that there is a set $A\subset A_\mathfrak{m}$ with $\card (A_\mathfrak{m}\setminus A)<\mathfrak{m}$
and $\pi_{A_\mathfrak{m}}^{-1}(A)\cap \pi_X^{-1}(U_i)\subset V_i$ for each $i$. Let $S=\{x_0\}\cup\bigcup_{i=0}^k S_i$. Now, we can
find an $\alpha\in A\setminus\{\mu\}$ with $\varphi(\alpha)=S$ (because $\card\varphi^{-1}(S)=\mathfrak{m}$).
$H(\alpha)=\pi_1(\{\alpha\}\times S\times Y)$ is homeo\-mor\-ph\-ic to $S\times Y$. Fix an index~$i$ for a while, and note that
$$\pi_1 (\{\alpha\}\times S_i\times Y)=\pi_{A_\mathfrak{m}}^{-1}(\alpha)\cap\pi_X^{-1}(S_i)\subset {\bigcap}_{\,0\leq j\leq k,\,j\neq i} V_j.$$
We claim that {\em no point of\/ $\pi_1 (\{\alpha\}\times\{x_0\}\times Y)$ belongs to $V_i$.} Indeed, $S_i$ converges to~$x_0$. If we
had $\pi_1(\alpha,x_0,y)\in V_i$, then there would exist a point~$x\in S_i$ such that $\pi_1 (\alpha,x,y)\in V_i$. In consequence, the
intersection of $V_i$'s would be non-empty, and $\mathcal{V}$~would not be a $\partial\Delta^k$-tuple. Therefore, $\pi_1
(\{\alpha\}\times\{x_0\}\times Y)=\pi_{A_\mathfrak{m}}^{-1}(\alpha)\cap\pi_X^{-1}(x_0)$ does not meet $V_i$ for any $i$, and is
contained in~$Q$. We can write $x^{\varphi(\alpha)}=x_0$.
\end{proof}

As a consequence of Theorem \ref{Indleq1} and the foregoing lemma we obtain

\begin{theo}\label{Indeq1}
Let\/ $k\geq 1$. Suppose that\/ $X$ and\/ $Y$ are non-empty compact spaces. If\/ $X$ is a Fr\'echet space, $\dstr{\partial\Delta^k} X=1$, and\/ $\dInd{\partial\Delta^k}X=\dInd{\partial\Delta^k} Y$,
then
\parbox{13cm}{$$\dInd{\partial\Delta^k} Z(X,Y)=\dInd{\partial\Delta^k} X+1.$$}\qed
\end{theo}

Lemma \ref{big-partition2} and Theorem \ref{Indeq1} hold for each simplicial complex $K$ (a similar proof with a more complicated
description of the set $\dobs{K}\mathcal{U}$ for arbitrary~$K$).

The following corollary results from the statements \ref{big-space1}, \ref{Indleq1}, and \ref{Indeq1}.

\begin{cor}\label{indeq2}
Let $k\geq 2$ and $n\geq 1$. If $X$ is a compact metric space such that\/ $\dim X=k(n+1)-1$, then
$$\dInd{\partial\Delta^k} Z(X,Y)=n+1 \mbox{ \ \ and \ \ }\dstr{\partial\Delta^k} Z(X,Y)=0$$
for every compact space $Y$ with $\dInd{\partial\Delta^k} Y=n$.\qed
\end{cor}

\begin{cor}\label{example1}
Let $k,n\geq 1$. If\/ $C$ is a metric continuum with\/ $\dim C=k(n+1)-1$, then $X_C=Z(C,C)$ is a Fr\'echet compact space such that
\begin{lister}
\item $\ddim{\partial\Delta^k} X_C=n$,
\item $\dInd{\partial\Delta^k} X_C=\dInd{|\partial\Delta^k|} X_C=n+1$,
\item $\dstr{\partial\Delta^k} X_C=0$ whenever $k\geq 2$,
\item every component of $X_C$ is homeomorphic to $C$.
\end{lister}
\end{cor}

\begin{proof}
By Proposition \ref{big-space1}, we obtain $\dInd{\partial\Delta^k} C=n$ and $\dstr{\partial\Delta^k} C=1$. The statement (a) results
from Theorem \ref{gen-relations} and Lemma \ref{ddim}, and (b) is a corollary to the statements \ref{Indeq1}, \ref{gen-relations}, and
\ref{ANR-leq}. Corollary \ref{indeq2} implies (c), and the application of Proposition \ref{general-zxy} completes the proof.
\end{proof}

\begin{rem}\em \label{rema}
(a) In the above Corollary \ref{example1}, the metrisable components $P$ of $Z(C,C)$ have $\dInd{\partial\Delta^k} P$~$=$
$\dInd{|\partial\Delta^k|} P$~$=$ $n$~$<n+1$~$=$ $\dInd{\partial\Delta^k} Z(C,C)$~$=$ $\dInd{|\partial\Delta^k|} Z(C,C)$. Thus,
$\dInd{\partial\Delta^k}$ and $\dInd{|\partial\Delta^k|}$ analogues of Theorem \ref{compon1} do not hold. This is no surprise because
there is not such an analogue for the large inductive dimension $\Ind$ (Chatyrko \cite{chat-08}; see also Krzempek~\cite{krz-10a}).

(b) Spaces similar to $Z(C,C)$ in Corollary \ref{example1} are constructed by Chatyrko~\cite{chat-08} for $k\!=\!n\!=\!1$ and
$C\!=\![0,1]$. The spaces have $\dim\! =\!1$, $\ind\!=\!\Ind\!=\!2$, and each of their components is either a singleton or a subspace
homeo\-morph\-ic to $[0,1]$. Also for $k=1$ and each integer $n>1$, similar spaces have been expected in \cite[Remark 5.1]{chat-08}.
We believe that if $X$ is a compact metric space with $\dim X=k(n+1)-1$, where $k,n\geq 1$, then $Z(X,X)$ contains compact subspaces
$Q\subset P$ such that $\dInd{\partial\Delta^k} Q=\dInd{|\partial\Delta^k|} Q=n$, $\dInd{\partial\Delta^k} P=\dInd{|\partial\Delta^k|}
P=n+1$, and $P\setminus Q$ is a discrete space of cardinality~$\mathfrak{c}$ (cf.\ \cite{chat-08}, a construction for $k=n=1$ and
$\Ind$).

(c) Suppose that $X$ is a compact metric space with $\dim X=k(n+1)-1$. Then $\ddim{|\partial\Delta^k|} X=n$, and $X$ is the union of
pairwise disjoint subspaces $X_0,\ldots,X_n$ with $\ddim{|\partial\Delta^k|} X_i=0$ for $i=0,\ldots,n$ (Fedorchuk \cite[Corollary
5.16]{fed-09}). Consider $Z(X,X)$ and its compact subspaces
$$Z_i=H(\mu)\cup\bigcup\{H(\alpha)\colon \varphi(\alpha) \mbox{ is finite or its unique accumulation point is in $X_i$}\}$$
for $i=0,\ldots,n$. Evidently $Z(X,X)=Z_0\cup\ldots\cup Z_n$. We shall sketch a proof of the equalities $\dInd{\partial\Delta^k}
Z_i=\dInd{|\partial\Delta^k|} Z_i=n$ for $i=0,\ldots,k$. Therefore, the space $Z(X,X)$ with
$\dInd{\partial\Delta^k}Z(X,X)=\dInd{|\partial\Delta^k|}Z(X,X)=n+1$ is the union of $n+1$ closed subspaces $Z_i$ with
$\dInd{\partial\Delta^k} Z_i=\dInd{|\partial\Delta^k|} Z_i=n$. This is similar to the properties of several well-known spaces (for
instance, Lokucievski\u\i 's Example 2.2.14 in \cite{eng-95}, Chatyrko's spaces in~\cite{chat-08}, Charalambous and Chatyrko's
examples for the dimension $\Indo$ in~\cite{chch-05}).

We have $n=\dInd{\partial\Delta^k} X\leq \dInd{\partial\Delta^k} Z_i\leq\dInd{|\partial\Delta^k|} Z_i$. The dimension $\dIndo{M}$
modulo a simplicial complex $M$ [respectively: modulo an ANR $M$\/] is defined similarly as $\dInd{M}$---in order that $\dIndo{M} X
\leq\alpha$ we stipulate that the $M$-partition~$P$ in the statement \ref{def3}(b) [respectively: \ref{def4}(b')] is a zero set with
$\dIndo{M} P<\alpha$ (see \cite[p.\ 670]{chk-12}). It is easily shown by transfinite induction that $\dInd{M} \leq\dIndo{M}$, and
Theorem 1 in \cite{chk-12} may be summarised as follows: {\em $\dIndo{K}=\dIndo{|K|}$ for any simplicial complex~$K$ and all normal
spaces.} Thus, we have $n$~$\leq$ $\dInd{|\partial\Delta^k|} Z_i$~$\leq$ $\dIndo{|\partial\Delta^k|} Z_i$~$=$
$\dIndo{\partial\Delta^k} Z_i$. It is sufficient to show that $\dIndo{\partial\Delta^k} Z_i\leq n$.

We need the following claim: {\em For each closed $\partial\Delta^k$-tuple $\mathcal{F}$ of $X$, there exists a
$\partial\Delta^k$-partition~$P$ disjoint from $X_i$.} Indeed, Lemma 6 in~\cite{chk-12} directly yields a map $f\colon
\bigcup\mathcal{F}\to |\partial\Delta^k|$ with $F_j\subset f^{-1}(K_j)$ for $j=0,\ldots,k$ (see the definition of $K_j$'s before Lemma
\ref{extend} herein). By Fedorchuk's \cite[Proposition 2.7]{fed-10b}, there is a $\partial\Delta^k$-parti\-tion $P$ for $f$ disjoint
from $X_i$, and hence, $f$ has an extension $f'\colon X\setminus P\to |\partial\Delta^k|$.\linebreak Since the sets $K_j'=\{x\in
|\partial\Delta^k|\colon x_j>0\}$ form a $\partial\Delta^k$-neigh\-bour\-hood $\mathcal{K}'$ of~$\mathcal{K}$, we can take the
pre-image $\partial\Delta^k$-tuple $f'^{-1}(\mathcal{K}')$. Thus, $P$ is a $\partial\Delta^k$-partition for~$\mathcal{F}$. Using the
above claim, remembering that each closed subset of $X$ is a zero subset, and modifying the proof of Theorem \ref{Indleq1}, one can
show that each closed $\partial\Delta^k$-tuple of $Z_i$ has a metrisable zero $\partial\Delta^k$-partition $P$ in $Z_i$ with
$\dInd{\partial\Delta^k} P$~$=$ $\dIndo{\partial\Delta^k} P<n$. This means that $\dIndo{\partial\Delta^k} Z_i\leq n$.

(d) Let $T\!=\!Z(C,C)$ be the space in Corollary \ref{example1}. If $k\!\geq\! 2$, then $\dstr{\partial\Delta^k} T\!=\!0$, and we
obtain $\dInd{\partial\Delta^k} Z(T,T)=n+1$ by Theorem \ref{Indeq1}. In the proof of Theorem \ref{ANR-main} we iterate the operation
$Z(X,Y)$. In the case of $\dInd{\partial\Delta^k}$ for $k\geq 2$, we do not know whether $\dstr{\partial\Delta^k} Z(T,T)=0$ or
$\dstr{\partial\Delta^k} Z(T,T)=1$. In consequence, {\em for $k\geq 2$ we do not know if the operation $Z(X,Y)$ allows us to construct
compact spaces $X$ with metrisable components and $\dInd{\partial\Delta^k} X>\ddim{\partial\Delta^k} X+1$.}
\end{rem}

\section{Conclusion and open problems}

The theories of inductive dimensions investigate problems which involve partitioning a given space in some admissible ways. The
following two questions arise. (1)~What closed subsets are sufficient or large enough to partition the space in all considered
circumstances/ways? (2)~How large closed subsets are necessary to partition the space? In the case of $\dInd{L}$ and $\dInd{K}$ of
$Z(X,Y)$, it is sufficient to consider $L$-partitions and $K$-partitions which are finite disjoint unions described by formulas
(\ref{1}) and~(\ref{2}) on pp.~\pageref{1} and~\pageref{2}. An answer to the latter question is stated by the alternatives (a) or (b)
of Lemmas \ref{Ind-raise1} and~\ref{big-partition2}.

In Sections \ref{fourth} and \ref{fifth} we have drawn up two maps of the $Z(X,Y)$ spaces' land. The difference between the maps has
enabled us to detect compact Fr\'echet spaces with $\dInd{\partial\Delta^k}<\dInd{|\partial\Delta^k|}$ (Corollary \ref{example2}). We
have found a quite exhaustive solution to the Problem stated in the Introduction in the case of $\dInd{L}$, where $L$ is a compact
metric ANR: for arbitrarily large ordinals $\alpha\geq n$, we have constructed compact Fr\'echet spaces with $\ddim{L}=n$,
$\dInd{L}=\alpha$, and all components metrisable (see Corollary \ref{ANR-cor} for necessary obstructions). In the case of $\dInd{K}$,
where $K$ is a finite simplicial complex, we have succeeded only for $K=\partial\Delta^k$ and $\alpha=n+1$ (Corollary \ref{example1}).
Crucial properties of $\dInd{\partial\Delta^k}$ and $\dstr{\partial\Delta^k}$ may be summarised as follows (Propositions
\ref{big-space1}, \ref{big-space2}, and Theorems \ref{Indleq1}, \ref{Indeq1}).

\begin{theo}
Let\/ $k,n\geq 1$ be integers. Suppose that\/ $X$ and\/ $Y$ are non-empty compact spaces,\/ $X$ is Fr\'echet, and\/
$\dInd{\partial\Delta^k} X=\dInd{\partial\Delta^k} Y$. Then the following implications hold.

\medskip

\noindent
\begin{minipage}[b]{30cc}
\begin{center}
\begin{tabular}{cc}
\parbox{15cc}{\center{$\dim X=kn$, where $k\geq 2$}} & \parbox{15cc}{\center{$\dim X=k(n+1)-1$}} \\
\parbox{15cc}{\center{$\Downarrow$}} & \parbox{15cc}{\center{$\Downarrow$}}\\
\parbox{15cc}{\center{$\dInd{\partial\Delta^k}X=n$ and\/ $\dstr{\partial\Delta^k} X=0$}} & \parbox{15cc}{\center{$\dInd{\partial\Delta^k}X=n$ and\/ $\dstr{\partial\Delta^k} X=1$}}\\
\parbox{15cc}{\center{$\Downarrow$}} & \parbox{15cc}{\center{$\Downarrow$}}\\
\parbox{15cc}{\center{$\dInd{\partial\Delta^k}Z(X,Y)=n$}} & \parbox{15cc}{\center{$\dstr{\partial\Delta^k}Z(X,Y)=0$ unless $k=1$,\\ and\/ $\dInd{\partial\Delta^k}Z(X,Y)=n+1$.}}
\end{tabular}
\end{center}
\end{minipage}

\vspace*{-.9em}\qed
\end{theo}

The specific question we are not able to answer is

\begin{quest}
Is it true that\/ $\dstr{\partial\Delta^2} Z([0,1]^2,[0,1]^2)=1$?
\end{quest}

An answer in the affirmative would give us hopes for finding a proof of the equality $\dstr{\partial\Delta^2} Z(T,T)=1$, where
$T=Z([0,1]^3,[0,1]^3)$. Having such a proof, we could apply Theorem \ref{Indeq1} to $X=Y=Z(T,T)$, and state a positive answer to

\begin{quest}
Do there exist a simplicial complex $K$ and a compact space $X$ such that the underlying polyhedron $|K|$ is connected, $\dInd{K} X
>\ddim{K} X+1$, and each component of $X$ is metrisable?
\end{quest}

\end{document}